\documentclass[aap,MSNbibl,dvips]{arximspdf}
\usepackage{mathbh}
\usepackage{graphicx}

\doi{10.1214/13-AAP931} 
\volume{24}
\issue{2}
\pubyear{2014}
\firstpage{616}
\lastpage{651}

\makeatletter

\newcommand{\rright}{\right}
\newcommand{\lleft}{\left}
\newtheorem{lemma}{Lemma}[section]
\newtheorem{teo}[lemma]{Theorem}
\newtheorem{prop}[lemma]{Proposition}
\newtheorem{cor}[lemma]{Corollary}
\newtheorem{lem}[lemma]{Lemma}
\newproclaim{oss}[lemma]{Remark}
\newtheorem{hyp}[lemma]{Hypothesis}
\newcommand{\dd}{\mathrm{d}}
\makeatother

\begin{document}
\begin{frontmatter}

\title{Stability of solitons under rapidly oscillating random
perturbations of the initial conditions}
\runtitle{Random perturbations of solitons}

\begin{aug}
\author[A]{\fnms{Ennio} \snm{Fedrizzi}\corref{}\ead[label=e1]{fedrizzi@math.univ-lyon1.fr}\thanksref{t1}}
\runauthor{E. Fedrizzi}
\affiliation{Universit\'e Paris Diderot}
\address[A]{LPMA, UMR 7599 CNRS\\
Universit\'e Paris Diderot\\
Sorbonne Paris Cit\'e\\
75205 Paris\\
France\\
\printead{e1}} 
\end{aug}
\thankstext{t1}{Supported by French Agence Nationale de la Recherche
(project MANUREVA ANR-08-SYSC-019).}

\received{\smonth{1} \syear{2012}}
\revised{\smonth{3} \syear{2013}}

%
\begin{abstract}
We use the inverse scattering transform and a diffusion approxi\-mation
limit theorem to study the stability of soliton components of the
solution of the nonlinear Schr\"odinger and Korteweg--de Vries
equations under random perturbations of the initial conditions: for a
wide class of rapidly oscillating random perturbations this problem
reduces to the study of a canonical system of stochastic differential
equations which depends only on the integrated covariance of the
perturbation. We finally study the problem when the perturbation is
weak, which allows us to analyze the stability of solitons
quantitatively.
\end{abstract}

%

\end{frontmatter}

\section{Introduction}\label{sec1}

The aim of the present work is to study the stabi\-lity of the soliton
components of solutions of completely integrable systems under rapidly
oscillating random perturbations of the initial condition. We will
consider and compare two important examples of equations widely
employed to model nonlinear and dispersive effects in wave propagation:
the (1-dimensional) nonlinear Schr\"odinger (NLS) equation,
%
%
\begin{equation}
\label{NLS intro} \frac{\partial U}{\partial t} + \frac{i}{2}\frac
{\partial^2 U}{\partial x^2}
+i|U|^2 U =0
\end{equation}
and the Korteweg--de Vries (KdV) equation,
%
%
\begin{equation}
\label{KdV intro} \frac{\partial U}{\partial t} + 6 U \frac{\partial
U}{\partial x} +
\frac{\partial^3 U}{\partial x^3}=0.
\end{equation}
The NLS equation models in particular short pulse propagation in
single-mode optical fibers (then $t$ is a propagation distance and $x$
is a time) \cite{MN}. The KdV equation models shallow water wave
propagation \cite{Wh}.

Explicit results are derived for the case of a square (box-like)
initial condition perturbed with a zero mean, stationary, rapidly
oscillating process $\nu(x/\varepsilon^2)$,
%
%
\begin{equation}
\label{IC intro} U_0 (x) = \biggl(q + \frac{\sigma}{\varepsilon} \nu \bigl(x/
\varepsilon^2 \bigr) \biggr) \mathbf{1}_{[0,R]}(x),
\end{equation}
but the results on the fast oscillating regime of Section~\ref{section
W} can be extended to the case of perturbation of a more general
initial condition defined by a bounded, compactly supported function
$q(x)$. The function $q$ must be real for KdV, but is allowed to take
complex values for NLS: the computations presented in Sections~\ref{section W} and \ref{sec NLS} are relative to the case of a real
$q$, but can be easily extended to the case of a complex (but with
constant phase) function $q$. The rapidly oscillating fluctuations of
the initial condition can model the high frequency additive noise of
the light source generating the pulse in nonlinear fiber optics, for
instance.

Our approach to both examples relies on the inverse scattering
transform (IST), a powerful tool used to study solutions of completely
integrable nonlinear equations; see \cite{APT}. In this framework, the
problem is transformed into a linear system of differential equations
where the initial condition enters as a
potential, and soliton components correspond to eigenvalues. %
Indeed, the solution of a nonlinear dispersive equation modeling the
propagation of waves may show two components with a very distinct
beha\-vior: the soliton components, composed of solitary waves that
propagate over arbitrarily large distances with constant velocity and
constant profile, and in addition, the radiation component, whose
amplitude decays in time as a~power law. The identification of the
soliton components therefore characterizes the long-time behavior of
the solution of the PDE. A short introduction to the IST is presented
in Section~\ref{sec:IST-det}, together with a discussion of the
deterministic case [$\sigma=0$ in equation (\ref{IC intro})].

We will show in Section~\ref{section W} that for rapidly oscillating
processes (small values of $\varepsilon$) the limit system governing
the stability of the soliton components reads as a set of stochastic
differential equations (SDEs), and it is formally equivalent to the
system where the initial condition contains a white-noise perturbation,
%
%
\begin{equation}
\label{IC intro W} U_0 (x) = (q + \sqrt{2\alpha} \sigma\dot
W_x ) \mathbf{1}_{[0,R]}(x),
\end{equation}
where $\alpha$ is the integrated covariance of the process $\nu$. This
shows that to study the soliton components in the limit of rapid
oscillations the only required parameter of the statistics of $\nu$ is
its integrated covariance. Notice that we cannot directly use a~white
noise to perturb the initial condition, as the IST requires some
integrability conditions on the initial condition (e.g., $U_0\in L^1$),
which are not satisfied by a white noise. The main result is presented
in Theorem \ref{teo convergenza Phi e Psi}.

We also obtain that solitons are stable under perturbations of the
initial condition for both examples studied; this is shown in Sections~\ref{sec NLS} and \ref{sec KdV}. However, a few interesting differences
will be pointed out; in particular, thresholding effects for the
creation of solitons are present in the NLS case and absent for KdV.

We also provide an easy way to compute the first order corrections to
the parameters characterizing the soliton components of the solution.

Results and some future directions of research are discussed in the
last section.

\section{The inverse scattering transform and the deterministic
problem}\label{sec:IST-det}

In the IST framework, a direct scattering problem (known as the
Zakharov--Shabat spectral problem, ZSSP) associated to the NLS equation
is introduced
%
%
\begin{equation}
\label{ZSSP} \cases{ \dfrac{\partial\psi_1}{\partial x} = i U_0(x) \psi_2
- i\zeta\psi_1,
\vspace*{6pt}\cr
\dfrac{\partial\psi_2}{\partial x} = i
U_0^*(x) \psi_1 + i\zeta\psi_2,}
\end{equation}
where $x\in\mathbb{R}$, $\psi_{i}(x)$, $i=1,2$, are the components of a
complex vector eigenfunction $\Psi(x)\in\mathbb{H}^1(\mathbb{R})$ and
$\zeta\in\mathbb{C}$ is the spectral parameter. The space
$\mathbb{H}^1(\mathbb{R})$ is defined as $\mathbb{H}^1(\mathbb{R})= \{
\Psi| \psi_i \in L^2(\mathbb{R}), \partial_x \psi_i \in
L^2(\mathbb{R}), i=1,2 \}$. When $U_0=0$, it is easy to see that the
continuous part of the spectrum is composed by the whole real line. The
eigenspace associated to the eigenvalue $\zeta\in\mathbb{R}$ has
dimension 2, and 
the functions
\[
\Psi\sim\pmatrix{ 1
\cr
0 } e^{-i\zeta x}, \qquad\Phi\sim\pmatrix{ 0
\cr
1 }
e^{i\zeta x}
\]
define a basis of this space. In this case, the discrete spectrum is
empty because the nontrivial solutions of $\partial_x f = i \zeta f$
are not in $L^2(\mathbb{R})$.


When introducing any localized initial condition $U_0$, by Weyl's
theorem the continuous spectrum (unlike the discrete spectrum) remains
unchanged, and for $\zeta\in\mathbb{R}$ the solutions $\Psi$,
$\widetilde\Psi$, $\Phi$, $\widetilde\Phi$ defined by the boundary
conditions
%
%
\begin{eqnarray}
\label{IC IST} \Psi &\sim&\pmatrix{ 1
\cr
0 } e^{-i\zeta x},\qquad\widetilde
\Psi\sim\pmatrix{0
\cr
1 } e^{i\zeta x},\qquad x\to-\infty,
\nonumber
\\[-8pt]
\\[-8pt]
\Phi&\sim&\pmatrix{ 0
\cr
1 } e^{i\zeta x},\qquad\widetilde\Phi\sim\pmatrix{1
\cr
0 } e^{-i\zeta x},\qquad x\to+\infty
\nonumber
\end{eqnarray}
produce two sets $\{\Psi, \widetilde\Psi\}$ and $\{\Phi, \widetilde
\Phi\}$ of linearly independent solutions. These functions are related
through the system
%
%
\begin{equation}
\label{eq:Psi-Phi matrix} \pmatrix{ \Psi
\vspace*{2pt}\cr
\widetilde\Psi} = \pmatrix{ b(
\zeta) & a(\zeta)
\vspace*{2pt}\cr
\widetilde a (\zeta) & \widetilde b (\zeta) } \pmatrix{ \Phi
\vspace*{2pt}\cr
\widetilde\Phi},
\end{equation}
where $a$, $b$ are called Jost coefficients and $\Psi$, $\Phi$ are
called Jost functions. The Jost coefficients are complex-valued
functions, while the Jost functions take values in~$\mathbb{C}^2$.
Therefore, products in the above equation have to interpreted as
``scalar times vector'' products, so that one has
\[
\pmatrix{ \psi_1
\cr
\psi_2 } = b(\zeta) \pmatrix{
\phi_1
\cr
\phi_2 } + a(\zeta) \pmatrix{ \widetilde
\phi_1
\vspace*{2pt}\cr
\widetilde\phi_2 }
\]
and similarly for $\widetilde\Psi$.

If $U_0\in L^1(\mathbb{R})$, the function $a(\zeta)$ can be
continuously extended to the upper half of the complex plane
$\mathbb{C}^+=\{\zeta\in\mathbb{C} | \Im[\zeta]>0\}$, where it is
analytic and can only have a countable number of simple zeros; see
\cite{APT}, Lemma~2.1. These zeros turn out to be the eigenvalues of
the discrete spectrum of the ZSSP (\ref{ZSSP}). If $\zeta_n$ is a zero
of $a$, then from (\ref{eq:Psi-Phi matrix}) we obtain that $\Psi$ and
$\Phi$ are linearly dependent. Due to (\ref{IC IST}) this implies that
the eigenfunction $\Psi_n$ relative to the eigenvalue $\zeta_n$ has an
exponential decay both at $-\infty$ and $+\infty$.

The IST is a powerful tool which allows us to solve many nonlinear
completely integrable systems, and the introduction of the IST
formalism is particularly convenient when dealing with soliton
components of the solution, as solitons have a very easy representation
in terms of the scattering variables: each zero of the Jost coefficient
$a$ in the upper complex half-plane ($\zeta=\xi+i\eta$, $\eta>0$)
corresponds to a soliton component of the solution. As we have just
remarked, these zeros of the Jost coefficient $a$ correspond to the
discrete spectrum of (\ref{ZSSP}).

When we study eigenfunctions of the spectral problem (\ref{ZSSP}) with
a potential $U_{0}$ of compact support in $[0,R]$ we can rewrite the
system (\ref{ZSSP}), which is defined for \mbox{$x\in\mathbb{R}$,} as a
system defined for $x\in[0,R]$, with some boundary conditions in $x=0$
and $x=R$ obtained from (\ref{IC IST}). This can be done as follows. By
inspecting the spectral problem (\ref{ZSSP}) we see that if
$\zeta\in\mathbb{C}^+$ is a discrete eigenvalue, then the corresponding
eigenfunction $\Psi\in\mathbb{H}^1(\mathbb{R})$ for $x\le
\nolinebreak0$ is given by $\psi_1(x) = e^{-i\zeta x}$ and $\psi_2(x)
=0$. For $x\ge R$ it must satisfy $\partial_x \psi_1 = -i \zeta\psi_1$
and $\partial_x \psi_2 = i \zeta\psi_2$, that is, to say $\psi_1(x) =
\psi_1(R) e^{-i \zeta(x-R)}$ and $\psi_2(x) = \psi_2(R) e^{i\zeta
(x-R)}$. Since the eigenfunction $\Psi$ must be integrable and
$\Im[\zeta]=\eta>0$, this implies that
\[
\psi_1(R)=0.
\]

A pure soliton solution of the NLS equation has the form
\[
\label{eq soliton NLS} U (t,x) = 2 i \eta\frac{\exp( -2i\xi x
-4i(\xi^2-\eta^2)t )}{\cosh( 2\eta(x+4\xi t) )},
\]
up to a shift and a phase.

Similarly, one can link the existence of soliton components of
solutions of the KdV equation to the spectral properties of an
associated equation, the first equation of the Lax pair,
%
%
\begin{equation}
\label{KdV lax 1} \frac{\partial^2 \varphi}{\partial x^2} + \bigl(U_0 +
\zeta^2 \bigr) \varphi=0,
\end{equation}
where the real function $\varphi(x)$ belongs to the Sobolev space
$W^{2,2}(\mathbb{R})$. We need to assume that
%
%
\begin{equation}
\label{integrability cond KdV} U_0 \in P_1:= \biggl\{ f
\dvtx \mathbb{R}\to\mathbb{R} \bigg| \int_{-\infty}^{\infty}
\bigl(1+|x| \bigr) \bigl| U(x) \bigr|\, \dd x <\infty \biggr\};
\end{equation}
see \cite{AC}, Chapter~2. 
Consider the continuous part of the spectrum of equation (\ref{KdV lax
1}), which is again the real axis. For $\zeta\in\mathbb{R}$, there are
two convenient complete sets of bounded functions solutions of
(\ref{KdV lax 1}), defined by their asymptotic behavior:
\begin{eqnarray*}
\phi(x,\zeta) &\sim& e^{-i\zeta x}, \qquad\widetilde\phi(x,\zeta) \sim
e^{i\zeta x}\qquad\mbox{for } x\to-\infty;
\\
\psi(x,\zeta) &\sim& e^{i\zeta x}, \qquad\widetilde\psi(x,\zeta) \sim
e^{-i\zeta x}\qquad\mbox{for } x\to+\infty.
\end{eqnarray*}
It follows from the above definitions that
\[
\phi(x,\zeta) = \widetilde\phi(x, -\zeta), \qquad\psi(x,\zeta) =\widetilde
\psi(x, -\zeta)
\]
and
\begin{eqnarray*}
\phi(x,\zeta) &=& a(\zeta) \widetilde\psi(x,\zeta) + b(\zeta) \psi (x,\zeta),
\\
\widetilde\phi(x,\zeta) &=& -\widetilde a (\zeta) \psi(x,\zeta) + \widetilde b(
\zeta) \widetilde\psi(x,\zeta).
\end{eqnarray*}
The function $a$ can be continuously extended to the upper half of the
complex plane~$\mathbb{C}^+$, where it is analytic and can only have a
finite number of simple zeros located on the imaginary axis
$\zeta=i\eta$; see \cite{AC}, Lemma 2.2.2. These zeros are the
eigenvalues of the discrete spectrum, and they correspond to the
soliton components of the solution.
%
%
A pure soliton solution is given by
\[
U (t,x) = 2\eta^{2}\operatorname{sech}^{2}
\bigl(\eta \bigl(x-x(t) \bigr) \bigr),
\]
where $x(t)=x_0 + 4\eta^2 t$ is the center of the soliton.

\subsection{NLS---deterministic box-shaped initial conditions} \label
{subsec NLS det}

Let the initial condition of the NLS equation be given by $U_0(x)= q
\mathbf{1}_{[0,R]}(x)$. Burzlaff proved in \cite{Bu88} that in this
case the number of solitons generated is the integer part of
$1/2+qR/\pi$; see also the relevant discussion and generalization of
\cite{Ki89}. They remark how physical intuition suggests that the first
soliton created when increasing $R$ corresponds to $\zeta=0$ (this is a
single soliton with zero amplitude and velocity, the \textit{quiescent
soliton}); this ``soliton'' is created for $qR=\pi/2$. For values of
$qR$ just over this critical threshold the created soliton has zero
velocity and nonzero amplitude $2\eta$ which can be computed explicitly
solving (\ref{ZSSP}) for pure imaginary values of $\zeta$.

In the first part of this subsection we report some computations
relative to this case, as the results and explicit formulas will be
used below. We then conclude the subsection providing the sketch of an
analytical proof of the claimed fact that generated solitons correspond
to purely imaginary values of $\zeta$.

When the potential $U_{0}$ is of compact support in $[0,R]$ and for
purely imaginary values of $\zeta=i\eta$, from the decaying condition
of the Jost function $\Psi$ at $-\infty$, one obtains the initial
condition
%
%
\begin{equation}
\label{NLS - CI} \Psi(0) = \pmatrix{ 1
\cr
0 } e^{\eta x} \bigg|_{x=0}
= \pmatrix{ 1
\cr
0 }.
\end{equation}
The system (\ref{ZSSP}) for $x\in[0,R]$ reads
%
%
\begin{equation}
\label{ZSSP z imm q re} \cases{\displaystyle\frac{\partial\psi_1}{\partial x} = i q
\psi_2 -i \zeta\psi_1,\vspace*{6pt}
\cr
\displaystyle
\frac{\partial\psi_2}{\partial x} = i q\psi_1 +i\zeta\psi_2}
\end{equation}
and $\Psi=(\psi_1, \psi_2)$ is a solution of the initial value problem
for $\zeta\neq iq$ if
%
%
\begin{eqnarray}
\psi_1 (x) &=& -\frac{i\zeta}{\sqrt{q^2+\zeta^2}} \sin \bigl( \sqrt{q^2+
\zeta^2} x \bigr) + \cos \bigl( \sqrt{q^2+
\zeta^2} x \bigr),\label{NLSE det sq wall - sol1a}
\\
\psi_2 (x) &=& i \frac{q}{\sqrt{q^2+\zeta^2}}\sin \bigl( \sqrt{q^2+
\zeta^2} x \bigr). \label{NLSE det sq wall - sol1b}
\end{eqnarray}
To be an eigenfunction, $\Psi$ needs to be integrable and to satisfy
the final condition $\psi_{1}(R)=0$ at $R$. This condition can be
rewritten for $\zeta\neq0$ as
%
%
\begin{equation}
\label{NLSE det sq wall f} f=\tan \bigl(\sqrt{q^2+
\zeta^2} R \bigr) + i \frac{\sqrt{q^2+\zeta^2}}{\zeta} =0.
\end{equation}
Since $a(\zeta)=\psi_1(R, \zeta) e^{i\zeta R} $, the function $f$ is
linked to the first Jost coefficient $a$ by the relation
\[
f(\zeta)=i a(\zeta) \frac{e^{-i\zeta R}}{\zeta} \frac{\sqrt{q^2+\zeta
^2}}{\cos( \sqrt{q^2+\zeta^2} R )}
\]
from which we see that the zeros of $f$ coincide with those of $a$,
except for \mbox{$\zeta=iq$.} However, for $\zeta=iq$ it is possible to
compute explicitly the solution of (\ref{ZSSP z imm q re}) satisfying
the initial conditions, which is given by
\[
\Psi=\pmatrix{ 1+x
\cr
i x }.
\]
Since this function does not satisfy the final conditions, no soliton
can be created for this particular value of $\zeta$.

To prove in an analytic way that the first soliton component of the
solution corresponds to a purely imaginary value of $\zeta$, we can
proceed as follows.

Recall that the zeros of $f$ coincide with those of $a$, and observe
that the function $a(\xi,\eta,R)$ is analytic in the domain $\mathbb{R}
\times(0,\infty) \times(0,\infty)$ and continuous in $\mathbb{R}
\times[0,\infty) \times(0,\infty)$. We use the argument principle to
study how the number of zeros in the upper half of the complex plane
evolves with increasing $R$. For any fixed $R$ we proceed as in
\cite{DP08}, taking a loop $C$ in the complex $\zeta$-plane composed of
the (lower) real axis and the infinite semi-arc in the upper half
plane. Then the number of zeros is given by
\[
N=\frac{1}{2\pi} \int_C \frac{1}{a}
\frac{\partial a}{\partial\zeta}\, \dd\zeta.
\]
Since $a=1+O(1/\zeta)$ for $|\zeta|\gg1$, the integral over the upper
part of the loop is zero. Changing variables $a(\zeta)=\rho(\zeta) \exp
(i\alpha(\zeta) )$, after some computations one obtains that unless
there is a zero on the real axis, also the integral on the lower part
of the loop is zero. Therefore, the number of zeros changes for a given
$R$ only if $a(\xi,0,R)=0$ admits a solution. But zeros of $a$ and $f$
coincide, and since in equation (\ref{NLSE det sq wall f}) for real
values of $\zeta=\xi\neq0$ the tangent is real, and the second term is
purely imaginary and nonzero, solutions of $f(\xi,0,R)=0$ can only be
found at $\xi=0$.

Explicit computations easily show that $\zeta=0$ corresponds to a
soliton solution only for $R=\frac{2n+1}{2q}\pi$, $n\in\mathbb{N}$.
Computing
explicitly the derivative of $a(\zeta)$ at $\zeta=0$ we get
\begin{eqnarray*}
\partial_\zeta a(\zeta) &=& e^{i\zeta R} \biggl[ \biggl( 2
\frac{\zeta R}{ \sqrt{q^2+\zeta^2} } - i \frac{q^2}{\sqrt{q^2+\zeta^2}
^3 } \biggr) \sin \bigl(
\sqrt{q^2+\zeta^2} R \bigr)
\\
&&\hspace*{92pt} {} + i R \frac{q^2 }{q^2+\zeta^2 } \cos \bigl( \sqrt{q^2+
\zeta^2} R \bigr) \biggr],
\\
\partial_\zeta a(\zeta) |_{\zeta=0} &= & - i \frac{1}{q}
\sin(q R) + i R \cos(q R),
\end{eqnarray*}
so that for $Rq = \frac{2n+1}{2}\pi$ the derivative is equal to
$(-1)^{n+1} i/q$ and is never zero. Therefore, new solitons are
generated one at a time, and they are immediately pushed (as $R$
increases) toward the interior of the domain. Karpman \cite{Ka79}
showed that if $a(\zeta)=0$, then $a'(\zeta)\neq0$; from this fact it
follows that zeros in the interior of the domain are always simple.
Considering the complex conjugate $\Psi^*$, which is a solution
whenever $\Psi$ is, one obtains that zeros not laying on the imaginary
axis always come in pairs $\pm\xi+i\eta$. But since zeros move
continuously (as $R$ grows) in the upper complex plane, cannot coalesce
and cannot leave the imaginary axis ($\xi=0$) unless they form a pair,
we get that they must remain on the imaginary axis.

\subsection{KdV---deterministic box-shaped initial conditions} \label
{subsec KdV det}

In \cite{Mu78}, Murray obtained a $\mathcal{C}^{\infty}$ solution for
the KdV equation with a deterministic ``box-shaped'' initial condition
$U_{0}=q \mathbf{1}_{[-R,R]}(x)$. He showed that in this case the Jost
coefficient $a$~extends to an analytic function in the upper part of
the $\zeta$-plane. Only in the case of positive values of $q$, $a$ has
a finite number of zeros on the imaginary axis $\zeta=i\eta$ for
$0<\eta\le\sqrt{q}$.

We report some explicit computations on our similar deterministic case,
as the results will be used below, and study some properties of the
soliton components of the solution.

First, let us construct explicitly the eigenfunctions solution of the
deterministic equation, which we call $\varphi_0$. Take here $U_0=q
\mathbf{1}_{[0,R]}(x)$ for some $q>0$. Due to \cite{AC}, Lemma 2.2.2,
we can assume that the eigenvalue is given by $\zeta=i\eta$; we need to
solve
%
%
\begin{equation}
\label{varphi_xx} \varphi_{xx}=\cases{\eta^{2}
\varphi, &\quad $x<0$, $x>R$, \vspace*{2pt}
\cr
\bigl(-q +\eta^{2} \bigr)
\varphi, &\quad $x\in[0,R]$.}
\end{equation}
For $\eta=0$ the only integrable solution is $\varphi\equiv0$.
Eigenfunctions corresponding to $\eta>0$ must satisfy
%
%
\begin{eqnarray}
\varphi &=& c_{1} e^{\eta x},\qquad x<0, \label{CI varphi}
\\
\varphi &=& c_{2} e^{-\eta x},\qquad x>R \label{CF varphi}.
\end{eqnarray}
For $x\in[0,R]$ one can rewrite the problem as
\[
\cases{ \partial_x \varphi= \widetilde\varphi,\vspace*{2pt}
\cr
\partial_x \widetilde\varphi= \bigl(-q + \eta^{2} \bigr)
\varphi.}
\]
Note that for $\eta\ge\sqrt{q}$ the solution of (\ref{varphi_xx}) is
monotone, so that it cannot be a Jost function corresponding to a
soliton (which has to be integrable). We therefore look for solutions
corresponding to $0<\eta< \sqrt{q}$. Set $c= \sqrt{q-\eta^2}$. Due to
(\ref{CI varphi}) and (\ref{CF varphi}), we only need to solve
(\ref{varphi_xx}) for $x\in[0,R]$. From (\ref{varphi_xx}) and the
initial conditions
\[
\varphi_0(0) = c_{1}, \qquad\partial_{x}
\varphi_0(0) = \eta c_{1}
\]
derived from (\ref{CI varphi}), we get
\[
\varphi_0(x) = \alpha e^{icx} + \beta e^{-icx},
\qquad\alpha= \frac{c_1}{2} \biggl(1-i \frac{\eta}{c} \biggr), \qquad
\beta= \frac{c_1}{2} \biggl(1+i \frac{\eta}{c} \biggr),
\]
which is to say
%
%
\begin{equation}
\label{varphi0} \varphi_0 (x) = c_1 \cosh(i cx) -i
c_1 \frac{\eta}{c} \sinh(icx) = c_1 \biggl[ \cos(c x)
+ \frac{\eta}{c} \sin(c x) \biggr].
\end{equation}
We can set the global constant $c_1$ equal to $1$. Matching this
solution with the final condition (\ref{CF varphi})
\[
\cases{\displaystyle\varphi(R) = \cos(cR) + \frac{\eta}{c} \sin(cR) =
c_2 e^{-\eta R}, \vspace*{2pt}
\cr
\displaystyle
\partial_x \varphi(R) = -c \sin(cR) + \eta\cos(cR) = -\eta
c_2 e^{-\eta R},}
\]
we obtain an equation for $\eta$,
\[
\cases{ q \sin \bigl(R\sqrt{q-\eta^2} \bigr) = 2 c_2
\eta\sqrt{q-\eta^2} e^{-R \eta},\vspace*{4pt}
\cr
\cos \bigl(R
\sqrt{q- \eta^2} \bigr) = c_2 \bigl(q-2\eta^2
\bigr) e^{-R\eta}.}
\]
For $\eta=\sqrt{q/2}$, the only possible solution is such that
$R\sqrt{q-\eta^2}=\pi/2 + k\pi$, which means that
%
%
\begin{equation}
\label{sqrt(q/2)} \sqrt{q} R=(2k+1) \pi/\sqrt{2}.
\end{equation}
All other solutions can be found solving
%
%
\begin{equation}
\label{eq tan-f} f(\eta):=\tan \bigl(R\sqrt{q-\eta^2} \bigr) -
\frac{ 2 \eta\sqrt{q-\eta^2}}{q-2\eta^2} =0
\end{equation}
for $\eta\in[0,\sqrt{q})\setminus\{\sqrt{q/2}\}$. The existence and the
number of solutions for the above equation depend on the quantity
$R\sqrt{q}$. Consider some fixed value of $R$. As it is shown below,
for small values of $q$ a first soliton is created with
$\eta^{(1)}\sim0$. As $q$ increases, the value of $\eta^{(1)}$
increases too and tends to $\sqrt{q/2}$ as $q$ tends to $\pi^2/(2R^2)$.
We have already found the solution for this specific value of $q$
[$k=0$ in equation~(\ref{sqrt(q/2)})]. For $q$ larger than
$\pi^2/(2R^2)$, $\eta^{(1)}$ continues to grow. A second solution
appears $(\eta^{(2)}=0)$ when $q=\pi^2/R^2$. A third solution appears
at $q=(2\pi/R )^2$; the values of $\eta^{(i)}$ (corresponding to the
$i$th soliton created) continuously increase as $q$~grows, but remain
ordered: $\eta^{(i)}<\eta^{(j)}$ for $i>j$. Therefore, the number of
solitons created is $ \lfloor R\sqrt{q}/\pi\rfloor+1$.
%

A few examples of $f(\eta)$ are plotted in Figure~\ref{fig:KdV}. We
have taken $R=1$ and different values of $q$. The first
critical\vspace*{-1pt} points (when new solitons are created)
correspond here to $q^{(2)}=\pi^2\sim 9,87$,
$q^{(3)}=4\pi^2\sim39,48$,\vspace*{1pt} $q^{(4)}=9\pi^2\sim88,83$,
$q^{(5)}=16\pi^2$. The almost-vertical line appearing near $\eta
=\sqrt{q/2}$ for $q=44$ reflects the fact that we are near the critical
points of (\ref{sqrt(q/2)}): from (\ref{sqrt(q/2)}) for $k=1$ we have
$q=9\pi^2/2\sim44.41$.
%
\begin{figure}

\includegraphics{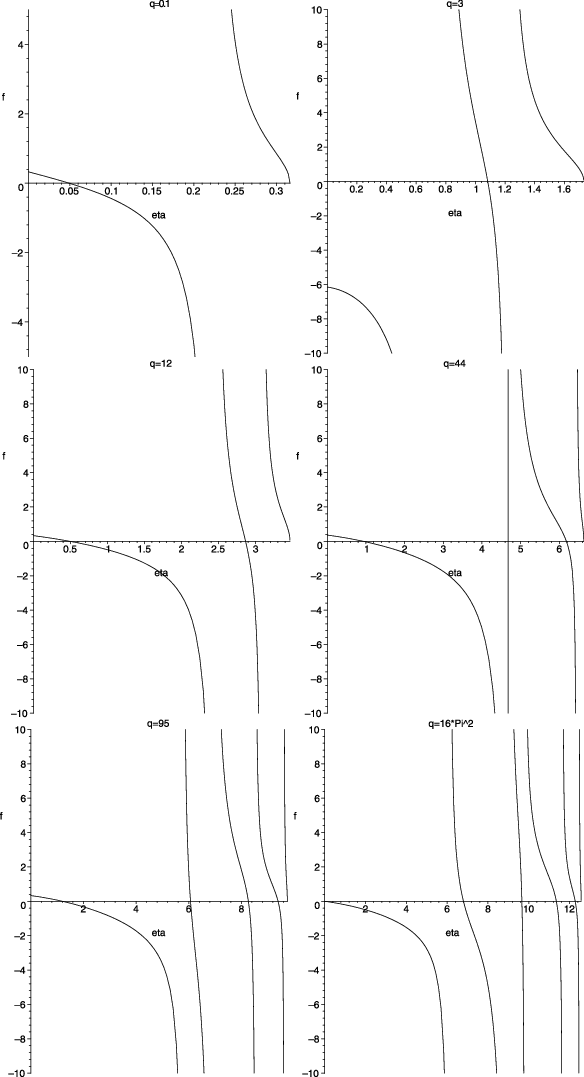}

\caption{Plot of the function $f(\eta)$ for different values of the
amplitude $q$ of the initial condition.~Each zero of $f(\eta)$
corresponds to a soliton component of the solution identified by the
complex number $i\eta$.} \label{fig:KdV}
\end{figure}

Let us take a closer look at the case $q\to0$. We assume $q=q_0
\varepsilon$ and look for the first terms of the expansion of $\eta$ in
$\varepsilon$: $\eta=\eta_0 + \eta_1 \varepsilon+ \eta_2
\varepsilon^2+O(\varepsilon^3)$. For finite values of $R$, by the above
considerations on the threshold effect the first term of the expansion
must be zero. Indeed, if we consider the expansion in $\varepsilon$ of
the function $f$ defined by (\ref{eq tan-f}), we obtain
{\renewcommand{\theequation}{ord. \arabic{equation}}
\setcounter{equation}{-1}
%
\begin{equation}
f(\eta) = \tan \Bigl(R\sqrt{- \eta_0^2} \Bigr)
+ \frac{\sqrt{-\eta_0^2} }{\eta_0}+ O(\varepsilon)
\end{equation}
and the order-zero term cannot be made equal to zero. We have therefore
$\eta=\eta_1\varepsilon+ O(\varepsilon^2)$. Looking at equation
(\ref{eq tan-f}) at first order
%
%
\begin{equation}
f(\eta) = ( R q_0 - 2\eta_1 ) \frac{\sqrt{\varepsilon}}{\sqrt{q_0}} + O
\bigl(\varepsilon^{3/2} \bigr) = 0,
\end{equation}
we obtain $\eta_1= \frac{R q_0}{2}$. Pushing the expansion further,
one can obtain the following order coefficients. At order two we have
%
%
\begin{equation}
f(\eta) = - \biggl( \frac{4}{24}R^3 q_0^2
+ 2 \eta_2 \biggr)\frac{\varepsilon^{3/2}}{\sqrt{q_0}} + O \bigl(\varepsilon^{5/2}
\bigr).
\end{equation}}\setcounter{equation}{20}%
One has then
%
%
\begin{equation}
\label{develop eta} \eta= \frac{R q_0}{2} \varepsilon- \frac{R^3 q_0^2}{12}
\varepsilon^2 + O \bigl(\varepsilon^3 \bigr)
\end{equation}
showing that in the limit $\varepsilon\to0$, $\eta$ is of the same
order of $q$.

\section{Limit of rapidly oscillating processes}\label{section W}
This section contains a rigorous justification of the use of the IST
when the initial condition contains a rapidly oscillating process. As
remarked in the \hyperref[sec1]{Introduction}, to be able to apply the IST, the initial
condition $U_0$ needs to satisfy some integrability conditions, $L^1$
for NLS and~(\ref{integrability cond KdV}) for KdV. For any
$\varepsilon>0$ these hypotheses are satisfied by initial conditions of
the form (\ref{IC intro}) if $\nu$ is bounded. Our objective is to show
that the IST applied to these random initial conditions gives a problem
that reads as a canonical system of SDEs in the limit $\varepsilon\to
0$. Thanks to the convergence result of Theorem \ref{teo convergenza
Phi e Psi} below, this limit system can be used to study the behavior
of rapidly oscillating initial conditions ($0<\varepsilon\ll1$), as we
shall do in the following sections. We stress that our interest is in
the study of rapidly oscillating initial conditions, which are
physically more relevant than the limit case of infinitely rapid
oscillations and for which the IST can be applied in a rigorous way. We
make the following assumptions (standard in the diffusion approximation
theory, \cite{FGPS}) on the process $\nu$:

%
\begin{hyp}\label{hyp nu}
Let $\nu(x)$ be a real, homogeneous, ergodic, centered, bounded, Markov
stochastic process, with finite integrated covariance\break
$\int_0^\infty\mathbb{E}[ \nu(0) \nu(x) ] \,\dd x = \alpha<\infty$ and
with generator $\mathcal{L}_\nu$ satisfying the Freedholm alternative.
\end{hyp}

Set
\[
U_0^\varepsilon:= \biggl(q + \frac{\sigma}{\varepsilon} \nu \bigl(x/
\varepsilon^2 \bigr) \biggr) \mathbf{1}_{[0,R]}(x)
\]
and note that for $x\in[0,R]$
\[
\int_0^x U_0^\varepsilon(y) \,
\dd y \stackrel{\varepsilon\to0} {\longrightarrow} \int_0^x
U_0(y) \,\dd y = qx + \sqrt{2\alpha} \sigma W_x
\]
in the space of continuous functions $\mathcal{C}^0 ([0,R];\mathbb{R}
)$, in distribution; see \cite{FGPS}. For every $\varepsilon>0$, we
apply the IST to the NLS and KdV equations with the initial condition
$U_0^\varepsilon$, and obtain the associated spectral problem. Then we
investigate the passage to the limit of this problem. We point out that
this passage to the limit is quite delicate: if it is relatively easy
to obtain a pointwise (in $\zeta$) convergence of the spectral data, to
obtain fine results for the limit case and for situations near the
limit case ($0<\varepsilon\le1$), a much stronger convergence is
needed.

We consider the ZSSP associated to the NLS equation: our goal is to
identify the points of the upper half of the complex plane, $\zeta\in
\mathbb{C}^+$, for which there exists a~solution $\Psi\in\mathbb {H}^1$
of the first order system (\ref{ZSSP}) for $x\in[0,R]$, satisfying the
boundary conditions
\[
\Psi(0)= \pmatrix{ 1
\cr
0 }\quad\mbox{and}\quad\psi_1(R)= 0
\]
derived from the exponentially decaying conditions (\ref{IC IST}).
These particular values of~$\zeta$ are the discrete eigenvalues of the
ZSSP and correspond to the soliton components.\vadjust{\goodbreak} The
strategy employed is to consider the flow $\Psi(x,\zeta)$, $x\in[0,R]$,
$\zeta\in \mathbb{C}^+$, solution of (\ref{ZSSP}) with initial
condition
%
%
\begin{equation}
\label{IC NLS} \Psi(0)= \pmatrix{ 1
\cr
0 }
\end{equation}
and look for the values of $\zeta$ for which the final condition is
satisfied.

For a fixed value of $\zeta$, we consider the solution
$\Psi^\varepsilon$ of the ZSSP obtained from the IST
%
%
\begin{equation}
\label{ZSSP approx} \cases{\displaystyle\frac{\partial\psi_1^\varepsilon
}{\partial x} =- i\zeta\psi_1^\varepsilon
+ i U_0^\varepsilon(x) \psi_2^\varepsilon,
\vspace*{6pt}
\cr
\displaystyle\frac{\partial\psi_2^\varepsilon}{\partial x} = i \bigl(U_0^\varepsilon(x)
\bigr)^* \psi_1^\varepsilon+ i\zeta\psi_2^\varepsilon}
\end{equation}
with initial condition
\[
\Psi^\varepsilon(0)= \pmatrix{ 1\cr 0 }.
\]
Now, \cite{FGPS}, Theorem 6.1, states that the process
$\Psi^\varepsilon$ converges in distribution in $\mathcal{C}^0 ([0,R];
\mathbb{C}^2)$ to the process $\Psi$ solution of
%
%
\begin{equation}
\label{NLS Ito} \cases{ \dd\psi_1 = \bigl[ \bigl(- i\zeta-\alpha
\sigma^2 \bigr) \psi_1 + i q \psi_2 \bigr] \,
\dd x + i\sqrt{2\alpha} \sigma\psi_2 \,\dd W_x,
\vspace*{4pt}
\cr
\dd\psi_2= \bigl[ i q \psi_1 + \bigl(i
\zeta-\alpha \sigma^2 \bigr) \psi_2 \bigr] \,\dd x + i
\sqrt{2\alpha} \sigma\psi_1 \,\dd W_x }
\end{equation}
with initial condition (\ref{IC NLS}). System (\ref{NLS Ito}) can be
rewritten in Stratonovich form as
%
%
\begin{equation}
\label{NLS Strat} \dd\Psi= i \pmatrix{ -\zeta& q
\cr
q & \zeta} \Psi\,\dd x + i
\sqrt{2\alpha} \sigma\pmatrix{ 0 & 1
\cr
1 & 0 } \Psi\circ \dd W_x.
\end{equation}

For NLS we can also consider perturbations produced by a complex
process: let $\nu_1, \nu_2$ be two independent copies of the process
$\nu$ and set $\widetilde\nu:= \nu_1+ i \nu_2$. One can define
$U_0^\varepsilon$ using $\widetilde\nu$ instead of $\nu$; proceeding as
above, from the IST one obtains again system (\ref{ZSSP approx}), and
from \cite{FGPS}, Theorem 6.1, one gets that in this case the limit
process is the solution of
%
%
\begin{eqnarray}\label{NLS Strat complex}
\dd\Psi &=& i \pmatrix{ - \zeta& q
\cr
q & \zeta} \Psi\,
\dd x + i\sqrt{2\alpha} \sigma \pmatrix{ 0 & 1
\cr
1 & 0 } \Psi\circ\dd
W_x^{(1)}
\nonumber\\[-8pt]\\[-8pt]
&&{} - \sqrt{2 \alpha} \sigma\pmatrix{ 0 & 1
\cr -1 & 0 } \Psi\circ\dd W_x^{(2)},\nonumber
\end{eqnarray}
where the $W^{(i)}$ are two independent Wiener processes, and with
$\Psi$ having the same initial condition (\ref{IC NLS}).

We apply the same strategy to the KdV equation: the goal is to obtain
the values of $\zeta\in\mathbb{C}^+$ for which there exists a solution
$\varphi^\varepsilon$ of
%
%
\begin{equation}
\label{ZS KdV approx} \varphi_{xx}^\varepsilon+
\bigl(U_0^\varepsilon+ \zeta^2 \bigr)
\varphi^\varepsilon=0
\end{equation}
with the boundary conditions
\[
\varphi^\varepsilon(0)= 1, \qquad\varphi_x^\varepsilon(0) = -i
\zeta, \qquad\varphi_x^\varepsilon(R) -i\zeta
\varphi^\varepsilon(R) = 0.
\]
These conditions correspond to imposing exponential decay of the
solution at infinity, so that $\zeta$ is an element of the discrete
spectrum of the spectral problem (\ref{ZS KdV approx}).

Setting $\Phi^\varepsilon:= (\varphi^\varepsilon,
\varphi_x^\varepsilon)^T$ this equation can be transformed into
%
%
\begin{equation}
\label{eq:Phi-eps} \dd\Phi^\varepsilon= \pmatrix{ 0 & 1
\cr
-U_0^\varepsilon-\zeta^2 & 0 } \Phi^\varepsilon\,
\dd x
\end{equation}
%
with boundary conditions
\[
\qquad\Phi^\varepsilon(0) = \pmatrix{ 1
\cr
-i\zeta}, \qquad
\phi_2^\varepsilon(R) - i\zeta\phi_1^\varepsilon(R)=0.
\]
We consider the flow $\Phi^\varepsilon(x,\zeta)$, $x\in[0,R]$,
$\zeta\in\mathbb{C}^+$, defined by the above equation with only the
initial condition, and look for the values of $\zeta$ s.t. the final
condition is satisfied. Again by \cite{FGPS}, Theorem 6.1,
$\Phi^\varepsilon$ converges in distribution to the solution~of
%
%
\begin{equation}
\label{eq:Phi-lim} \dd\Phi= \pmatrix{ 0 & 1
\cr
-q-\zeta^2 & 0 }
\Phi\,\dd x + \sqrt{2\alpha} \sigma\pmatrix{ 0 & 0
\cr
1 & 0 } \Phi\,\dd
W_x,
\end{equation}
which, in terms of the function $\varphi$, can be rewritten as
%
%
\begin{equation}
\label{KdV limite Intro} \dd\varphi_{x} = - \bigl(q+
\zeta^2 \bigr) \varphi\,\dd x + \sqrt{2\alpha} \sigma\varphi\,\dd
W_x.
\end{equation}
The initial condition is
%
%
\begin{equation}
\label{IC KdV} \Phi(0) = \pmatrix{ 1
\cr
-i\zeta}\quad\mbox{or equivalently}
\quad\varphi(0)= 1, \qquad\varphi_x (0) = -i \zeta.
\end{equation}
We remark that in the last two differential equations above the
Stratonovich and It\^o stochastic integrals coincide.\vadjust{\goodbreak}

The convergence obtained above is only for a (finite number of) fixed
$\zeta$ and $\sigma$, but we will need a convergence in $\mathcal{C}^0
( [0,R]; \mathcal{C}^1( \mathbb{R}^3 ) )$ to be able to differentiate
the limit process with respect to the parameters. This is the main
result of this section and it is provided by the following theorem. We
will focus on the problem of finding the values of $\zeta$ for which
the limit flows $\Psi$ and $\Phi$ match the final conditions in
Sections~\ref{sec NLS} and \ref{sec KdV}.

%
\begin{teo}\label{teo convergenza Phi e Psi}
Assume Hypothesis \ref{hyp nu}. Let $\Psi^\varepsilon:=
(\psi_1^\varepsilon, \psi_2^\varepsilon)^T$ be the solution of
(\ref{ZSSP approx}) with initial condition (\ref{IC NLS}) and $\Psi$
the solution of (\ref{NLS Strat}) with the same initial condition. Let
also $\varphi^\varepsilon$ be the solution of (\ref{ZS KdV approx})
with initial condition (\ref{IC KdV}) and $\varphi$ be the solution of
(\ref{KdV limite Intro}) with the same initial condition. Considering
these as functions of the space variable $x$ and the parameters $\xi,
\eta, \sigma$, we have in the limit of $\varepsilon\to0$ that
$\Psi^\varepsilon(x, \xi,\eta,\sigma) \to\Psi(x, \xi,\eta,\sigma)$
weakly in $\mathcal{C}^0 ( [0,R]; \mathcal{C}^1( \mathbb{R}^3;
\mathbb{C}^2 ) )$ and $\varphi^\varepsilon(x, \xi,\eta,\sigma)
\to\varphi(x, \xi,\eta,\sigma)$ weakly in $\mathcal{C}^0 ( [0,R];
\mathcal{C}^1( \mathbb{R}^3; \mathbb{C}) )$.
\end{teo}

To prove this theorem we need the following standard tightness
criteria; see~\cite{Me}, Chapter~2. We will use $\mathcal{D}( [0,R]; E
)$ to denote the space of CadLag processes defined for $x\in[0,R]$ and
with values in the space $E$. The first lemma is due to Aldous,
\cite{Al78}.

%
\begin{lemma}\label{lemma Aldous}
Let $(E,d)$ be a metric space, and $X^\varepsilon$ a process with paths
in $\mathcal{D}( [0,R]; E )$. If for every $x$ in a dense subset of
$[0,R]$ the family $ ( X^\varepsilon(x) )_{\varepsilon\in(0,1]}$ is
tight in $E$, and $X^\varepsilon$ satisfies the Aldous property:
\begin{longlist}[A:]
\item[A:] For any $\kappa>0$ and $\lambda>0$, there exists $\delta>0$ s.t.
\begin{eqnarray}
\limsup_{\varepsilon\to0} \sup_{\tau<R} \sup
_{0<\theta<\delta\wedge(R-\tau)} \mathbb{P} \bigl( \bigl\| X^\varepsilon(\tau +\theta) -
X^\varepsilon(\tau) \bigr\| > \lambda \bigr) <\kappa
\nonumber
\\
\eqntext{\mbox{where $\tau$ is a stopping time};}
\end{eqnarray}
\end{longlist}
then the family $ ( X^\varepsilon)_{\varepsilon\in(0,1]}$ is tight in
$\mathcal{D}( [0,R]; E )$.
\end{lemma}

To state the next lemma, we need to introduce some notation. If
$\mathcal{H}$ is a Hilbert space, and $\mathcal{H}_n$ a subspace of
$\mathcal{H}$, we shall use $\pi_{\mathcal{H}_n}^{}$ to denote the
projection of $\mathcal{H}$ onto $\mathcal{H}_n$. Also, $d_\mathcal
{H}$ is used to denote the distance on $\mathcal{H}$ introduced by the
inner product. 

%
\begin{lemma}\label{lemma approx H}
Let $\mathcal{H}$ be a Hilbert space and $\mathcal{H}_n$ be an
increasing sequence
of finite-dimensional subspaces of $\mathcal{H}$ s.t., for any $h\in
\mathcal{H}$,
$\lim_{n\to\infty} \pi_{\mathcal{H}_n}^{} h = h$. Let $(
Z^\varepsilon)_{\varepsilon\in(0,1]}$ be a family of $\mathcal{H}$-valued
random variables. Then the family $( Z^\varepsilon
)_{\varepsilon\in(0,1]}$ is tight if and only if for any $\kappa>0$ and
$\lambda>0$, there exist $\rho_\kappa$ and a subspace
$\mathcal{H}_{\kappa,\lambda}$ s.t.
%
%
\begin{equation}
\label{approx H} \sup_{\varepsilon\in(0,1]} \mathbb{P} \bigl( \bigl\|
Z^\varepsilon\bigr\| \ge\rho_\kappa \bigr) \le\kappa\quad\mbox{and}\quad
\sup_{\varepsilon\in(0,1]} \mathbb{P} \bigl( d_\mathcal{H}
\bigl(Z^\varepsilon, \mathcal{H}_{\kappa, \lambda} \bigr) >\lambda \bigr) \le
\kappa.
\end{equation}
\end{lemma}

The proof of Lemma \ref{lemma approx H} can be found in \cite{Me84}.
For completeness we give it in \hyperref[sec:app]{Appendix}, together
with other technical results needed for the proof the Theorem~\ref{teo
convergenza Phi e Psi}.

\begin{pf*}{Proof of Theorem \ref{teo convergenza Phi e Psi}}
To unify notation, we shall use $X^\varepsilon$ to denote both
$\Psi^\varepsilon$ and $\Phi^\varepsilon$. Therefore, $X^\varepsilon$
is the solution of what we shall call the approximated system, which is
either system (\ref{ZSSP approx}) or (\ref{eq:Phi-eps}), with
$\varepsilon>0$.

Since Propositions \ref{prob reg zeta-sigma} and \ref{prob reg
eta-sigma} ensure that the limit equations for $\Psi$ and $\Phi$ have a
unique solution which is $\mathcal{C}^0 ( [0,R]; \mathcal{C}^1(
\mathbb{R}^3 ) )$, it suffices to prove convergence in the space of
CadLag processes $\mathcal{D}( [0,R]; \mathcal{C}^1( \mathbb{R}^3 ) )$.
We will do so in three steps.

Step~1 contains a technical result needed for the application in step~2
of Lemma~\ref{lemma approx H}, namely the proof of the bound (\ref
{approx H 2a}). 

In step 2, using Lemma \ref{lemma approx H}, we will show that for
every fixed $x$ the sequences $ (\Psi^\varepsilon(x) )_\varepsilon$ and
$ (\Phi^\varepsilon(x) )_\varepsilon$, denoted $ (X^\varepsilon(x)
)_\varepsilon$ in the following, are tight in the Hilbert space
$\mathcal{H}:=W^{3,2}(G)$. Note that, by Sobolev imbedding,
$\mathcal{H}\hookrightarrow\mathcal{C}^1(G)$. Here, $G$ is an open,
bounded subset of $\mathbb{R}^3$, the space of para\-meters. For
simplicity we take $G=(-N,N)^3$ for some real positive constant $N$; a
justification of the fact that it is not restrictive to assume that the
set of parameters $G$ is bounded is given below in the proof of
Proposition \ref{prob reg zeta-sigma}, where the convergences we are
proving here will be used.

In the last step we will use Lemma \ref{lemma Aldous}, where we take
$E$ to be the Hilbert space~$\mathcal{H}$. This will provide the
desired convergence of the family of processes $(X^\varepsilon
)_{\varepsilon\in(0,1]}$ in $\mathcal{D}( [0,R]; \mathcal{C}^1(
\mathbb{R}^3 ) )$.

Since $X^\varepsilon$ is the solution of a linear differential
equation 
with coefficients smooth in the parameters $\mu=(\xi, \eta, \sigma
)$, from the explicit formula for the solution, we get that
$X^\varepsilon(x,\mu)$ is smooth in the parameters. 
We will soon use its derivatives 
in the parameters: the vector of $X^\varepsilon$ and its first
derivatives in $\mu$ still satisfy a linear system of ODEs whose
coefficients depend linearly on the parameters and on the process $\nu
(x/\varepsilon^2)$, and the same result holds adding higher order derivatives.
\begin{longlist}[\textit{Step} 1]
\item[\textit{Step} 1 (A preliminary estimate).] The key point to
    show that $ ( X^\varepsilon(x,\mu) )_{\mu\in G}$ is tight in
    $\mathcal{H}=W^{3,2}(G)$ for every $x\in[0,R]$ is the proof of the
    bound
%
%
\begin{equation}\label{approx H 2a}
\limsup_{\varepsilon\to0} \mathbb{E} \bigl[ \bigl\|
X^\varepsilon(x, \cdot) \bigr\|_{W^{6,2}(G)} \bigr] \le C <\infty
\end{equation}
uniformly in $x\in[0,R]$. This is the content of this step.

Define $Y^\varepsilon$ as the vector process of $X^\varepsilon$ and all
of its derivatives in the parameters $\mu=(\xi,\eta,\sigma)$ up to
order 6. As remarked above, this process is the solution of a linear
system of ODEs with coefficients (the matrices $M_1$ and $M_2$) linear
in the parameters
\[
\frac{\dd}{\dd x} Y^\varepsilon= M_1 Y^\varepsilon+
\frac{1}{\varepsilon} \nu \bigl(x/\varepsilon^2 \bigr) M_2
Y^\varepsilon.
\]
Since $G$ is bounded, we only need to check that the second moment of
$Y^\varepsilon(x,\mu)$ is uniformly bounded with respect to
$\varepsilon\in(0,1]$, $\mu\in G$ and $x\in[0,R]$. Actually, we aim at
a stronger result, which we will need later. We are going to show that
[recall that $Y_{0}=Y_{0}^{\varepsilon}$ is deterministic since both
$\Psi^\varepsilon(0)$, $\Phi^\varepsilon(0)$ and their derivatives in
zero are defined by the equation for $x\le0$, which is deterministic,
and the boundary condition at $x\to-\infty$]
%
%
\begin{equation}
\label{stima unif Ye}
\mathbb{E} \Bigl[ \sup_{x\in[0,R]} \bigl|Y^{\varepsilon}(x) \bigr|^2 \Bigr] \le C_{R} \bigl(1+ \bigl|Y(0)
\bigr|^2 \bigr) < \infty.
\end{equation}
Following \cite{FGPS}, Section~6.3.5, we show this bound with the
perturbed test function method. Let $\mathcal{L}^\varepsilon$ be the
infinitesimal generator of the process $Y^\varepsilon$ and $\mathcal
{L}$ the infinitesimal generator of the process $Y$ obtained from the
process $X$ solution of the limit system (\ref{NLS Strat}) or
(\ref{eq:Phi-lim}) and its derivatives. Let $m\in N$ be such that
$Y^{\varepsilon}(x) \in\mathbb{C}^m$, and let $K$ be a compact subset
of $\mathbb{R}$ containing the image of the bounded process $\nu$. For
every $y\in\mathbb{C}^m$ and $z\in K$, $\mathcal{L}^\varepsilon$ has
the form
\[
\mathcal{L}^\varepsilon g(y,z) = \frac{1}{\varepsilon^2} \mathcal
{L}_\nu g(y,z) + \frac{z}{\varepsilon} (M_2 y
)^T \nabla_y g(y,z) + ( M_1 y
)^T \nabla_y g(y,z),
\]
where $\mathcal{L}_\nu$ is the infinitesimal generator of the process
$\nu$, as defined in Hypo\-thesis \ref{hyp nu}. Let $f$ be the identity
function on $\mathbb{C}^m$ and $f^{\varepsilon}(y,z) = y + \varepsilon
f_{1}(y,z)$ be the associated perturbed function, which is solution of
the Poisson equation $\mathcal{L}_{\nu} f_{1} (y,z) = - z ( M_2 y )^{T}
\nabla_y f(y)$. In this equation, $y$ plays the role of a frozen
parameter, so that $f_1$ has linear growth in $y$, uniformly in $z$,
and the same holds for
\[
\mathcal{L}^\varepsilon f^\varepsilon(y,z) = z ( M_2 y
)^T \nabla_y f_1(y,z) + \varepsilon(
M_1 y )^T \nabla_y f_1(y,z).
\]
Since
\begin{eqnarray*}
Y^{\varepsilon}(x) &= & Y^\varepsilon(0) - \varepsilon \bigl[
f_1 \bigl(Y^\varepsilon(x), \nu^\varepsilon(x) \bigr) -
f_1 \bigl(Y^\varepsilon(0), \nu^\varepsilon(0) \bigr) \bigr]
\\
&&{} + \int_0^x \mathcal{L}^\varepsilon
f^\varepsilon \bigl(Y^\varepsilon \bigl(x' \bigr),
\nu^\varepsilon \bigl(x' \bigr) \bigr) \,\dd x' +
M^\varepsilon_x,
\end{eqnarray*}
where $M^\varepsilon_x$ is a vector valued martingale, we get the bound
\begin{eqnarray*}
\sup_{x\in[0,R]} \bigl| Y^\varepsilon(x) \bigr| &\le& \bigl| Y^\varepsilon(0)
\bigr| + \varepsilon C \Bigl[ 1+ \sup_{x\in[0,R]} \bigl| Y^\varepsilon(x) \bigr|
\Bigr]
\\
&&{} + C \int_0^R 1+ \sup_{x' \in[0,x]}
\bigl| Y^\varepsilon \bigl(x' \bigr) \bigr| \,\dd x + C \sup
_{x\in[0,R]} \bigl| M^\varepsilon_x \bigr|.
\end{eqnarray*}
For $\varepsilon\le1/2C$, applying Gronwall's inequality and renaming
constants we get
%
%
\begin{eqnarray}\label{stima Ye 1}
\sup_{x\in[0,R]} \bigl| Y^\varepsilon(x) \bigr| &\le&
C_R \Bigl( 1+ \bigl| Y^\varepsilon(0) \bigr| + \sup_{x\in[0,R]}
\bigl| M^\varepsilon_x \bigr| \Bigr).
\end{eqnarray}
The quadratic variation of the martingale is given by
\[
\bigl\langle M^\varepsilon \bigr\rangle_x = \int
_0^x g^\varepsilon \bigl(Y^\varepsilon
\bigl(x' \bigr), \nu^\varepsilon \bigl(x' \bigr)
\bigr) \,\dd x',
\]
where
\begin{eqnarray*}
g^\varepsilon(y,z) &=& \bigl( \mathcal{L}^\varepsilon f^{\varepsilon
2} - 2
f^\varepsilon\mathcal{L}^\varepsilon f^\varepsilon \bigr) (y,z)
\\
&=& \bigl( \mathcal{L}_\nu f_1^2 -
2f_1 \mathcal{L}_\nu f_1 \bigr) (y,z)
\\
&&{} + 2\varepsilon z \bigl[ ( M_2 y )^T f_1
(y,z) - ( M_2 y )^T \bigl( (\nabla_y
f_1)^T f_1 \bigr) (y,z) \bigr]
\\
&&{} + 2 \varepsilon^2 \bigl[ ( M_1 y )^T
f_1 (y,z) - ( M_1 y )^T \bigl( (
\nabla_y f_1)^T f_1 \bigr) (y,z)
\bigr]
\end{eqnarray*}
has quadratic growth in $y$ uniformly in $z\in K$. Therefore, by Doob's
inequa\-lity,
\[
\mathbb{E} \Bigl[ \sup_{x\in[0,R]} \bigl| M^\varepsilon_x
\bigr|^2 \Bigr] \le C \mathbb{E} \bigl[ \bigl\langle M^\varepsilon
\bigr\rangle_R \bigr] \le C \int_0^R
1+ \mathbb{E} \bigl[ \bigl| Y^\varepsilon(x) \bigr|^2 \bigr] \,\dd x.
\]
Substituting into the expected value of the square of (\ref{stima Ye
1}) and using again Gronwall's inequality, we get (\ref{stima unif
Ye}), which gives (\ref{approx H 2a}).

\item[\textit{Step} 2 [{Tightness of $X^\varepsilon(x)$].}] In this
    step
    we show how to obtain the tightness of the family $ (
    X^\varepsilon(x,\mu) )_{\mu\in G}$ from (\ref{approx H 2a}) using
    Lemma \ref{lemma approx H}. Indeed, from this bound the first part
    of condition (\ref{approx H}) follows by the Markov inequality if
    we take $\rho_\kappa=C/\kappa$. The bound (\ref{approx H 2a})
    provides also information on the regularity of $X^\varepsilon(x)$,
    which can be used to prove the second part of condition
    (\ref{approx H}) as follows. By the Sobolev imbedding
    $W^{6,2}(G)\hookrightarrow\mathcal{C}^4(G)$ the bound (\ref{approx
    H 2a}) implies that $X^\varepsilon\in\mathcal {C}^4(G)$. An
    appropriate sequence of finite-dimensional subspaces $(\mathcal
    {H}_n)_n$ is constructed in the \hyperref[sec:app]{Appendix} in Lemma
    \ref{lemma:construc-Hn}, and Lemma \ref{lemma tecnico approx Hn}
    states that for any function $g\in\mathcal{C}^4(G)$ there exists an
    arbitrarily good approximation $g_n$ belonging to some~$\mathcal{H}_n$. Moreover, the control on the distance between $g$
    and the subspace $\mathcal{H}_n$ only depends on the norm
    $\|g\|_{\mathcal{C}^4}$, so that by (\ref{approx H 2a}) the
    approximation is uniform in~$\varepsilon$. This provides the second
    part of condition (\ref{approx H}). Finally, Corollary \ref{cor
    approx H} provides the last hypothesis of Lemma \ref{lemma approx
    H}, namely that for the sequence of subspaces $\mathcal{H}_n$
    constructed in Lemma \ref{lemma:construc-Hn} and for any
    $h\in\mathcal{H}$, $\lim_{n\to \infty} \pi_{\mathcal{H}_n}^{} h =
    h$. Then, Lemma \ref{lemma approx H} gives that for every $x$, the
    family $ ( X^\varepsilon(x,\mu) )_{\mu\in G}$ is tight in
    $\mathcal{H}=W^{3,2}(G)$.

\item[\textit{Step} 3 (Tightness of $X^\varepsilon$).] Thanks to Lemma
    \ref{lemma Aldous}, the tightness of the family of processes
    $X^\varepsilon$ in $\mathcal{D}( [0,R]; \mathcal{H} )$ follows if
    we show that the Aldous property~A holds. Since $G$ is
    bounded, we can prove the Aldous property showing that
%
%
\begin{equation}
\label{eq step 1 aldous} \lim_{\delta\to0} \limsup
_{\varepsilon\to0} \sup_{\mu\in G} \sup_{\tau\le R}
\sup_{0<\theta<\delta} \mathbb{E} \bigl[ \bigl| Y^{\varepsilon} (\tau+\theta,
\mu) - Y^\varepsilon(\tau,\mu) \bigr|^2 \bigr] =0,
\end{equation}
where $Y^\varepsilon$ is the vector process having as components
$X^\varepsilon$ and its derivatives in $\mu$ up to the third order
only. We prove the above limit using again the perturbed test function
method. With the notation introduced above, we have
\begin{eqnarray*}
&& \bigl| Y^\varepsilon(\tau+\theta) - Y^\varepsilon(\tau) \bigr|^2
\\
&&\qquad  \le C \bigl| M_{\tau+\theta}^\varepsilon- M^\varepsilon _{\tau}
\bigr|^2 + C \int_{\tau}^{\tau+\theta} \bigl|
\mathcal{L}^\varepsilon f^\varepsilon \bigl(Y^\varepsilon(x),
\nu^\varepsilon(x) \bigr) \bigr|^2 \,\dd x
\\
&&\quad\qquad{} + C\varepsilon \Bigl( 1+ \sup_{x\in[\tau,\tau+\theta]} \bigl|Y^\varepsilon(x)
\bigr|^2 \Bigr)
\\
&&\qquad \le C \bigl| M_{\tau+\theta}^\varepsilon- M^\varepsilon_{\tau}
\bigr|^2+ C \int_{\tau}^{\tau+\theta} \bigl|
\mathcal{L}^\varepsilon f^\varepsilon \bigl(Y^\varepsilon,
\nu^\varepsilon \bigr) - \mathcal{L}f \bigl(Y^\varepsilon \bigr)
\bigr|^2 \,\dd x
\\
&&\quad\qquad{} + C \int_{\tau}^{\tau+\theta} \bigl| \mathcal{L}f
\bigl(Y^\varepsilon \bigr) \bigr|^2 \,\dd x + C\varepsilon \Bigl( 1+
\sup_{x\in[\tau,\tau+\theta]} \bigl|Y^\varepsilon(x) \bigr|^2 \Bigr).
\end{eqnarray*}
We have that
\[
\mathbb{E} \bigl[ \bigl| M_{\tau+\theta}^\varepsilon- M^\varepsilon
_{\tau} \bigr|^2 \bigr] = \mathbb{E} \bigl[ \bigl(M_{\tau+\theta}^\varepsilon
\bigr)^2 - \bigl(M^\varepsilon_{\tau}
\bigr)^2 \bigr] = \mathbb{E} \biggl[ \int_{\tau}^{\tau+\theta}
\,\dd \bigl\langle M^\varepsilon \bigr\rangle_x \biggr].
\]
Since $ | \mathcal{L}^\varepsilon f^\varepsilon(y,z) - \mathcal
{L}f(y) |
\le\varepsilon C (1+ |y|)$ and $|\mathcal{L}f (y)| \le C|y|$, for
$\theta\le\delta$
\[
\mathbb{E} \bigl[ \bigl| Y^\varepsilon(\tau+\theta)- Y^\varepsilon(\tau)
\bigr|^2 \bigr] \le C_{R} (\delta+\varepsilon) \Bigl( 1+
\mathbb{E} \Bigl[\sup_{x\in[0,R]} \bigl|Y^\varepsilon(x) \bigr|^2
\Bigr] \Bigr).
\]
The right-hand side is independent of $\tau$, and we can use estimate
(\ref{stima unif Ye}) to bound it uniformly in $\varepsilon$ and $\mu$.
Therefore, (\ref{eq step 1 aldous}) follows, and the proof of Theorem
\ref{teo convergenza Phi e Psi} is complete.\quad\qed
\end{longlist}\noqed
\end{pf*}

%
\begin{oss} Using the notion of pseudo-generators, as introduced in
\cite{EK}, Section~7.4, 
it is possible to relax the conditions imposed on the driving
process~$\nu(x)$, assuming that it is just a mixing process instead of
Markov.
\end{oss}
%

\section{Stability of NLS solitons}\label{sec NLS}

This section is devoted to the study of our first example, the NLS
equation. We focus on the soliton components of the solution. In the
previous section we have obtained the limit equation (\ref{NLS Strat})
and we have remarked that every soliton component (soliton, in short)
is identified by a complex number $\zeta=\xi+i\eta$ s.t. the flow
$\Psi(x,\zeta)$ solution of (\ref{NLS Strat}) with initial condition
(\ref{IC NLS}) satisfies also a given final condition. The real and
imaginary parts of $\zeta$ define the velocity and amplitude of the
soliton, respectively. In Section~\ref{subsec NLS det} we presented
some classical results on the background deterministic solution. We
analyze now how this solution is modified by the introduction of a
real, in Section~\ref{subsec NLS W real}, or complex, in Section~\ref{subsec NLS W complex}, small-amplitude white noise perturbation of
the initial condition. The main results are contained in Propositions
\ref{prop NLS 1} and \ref{prop NLS complex}. We deal with the limit
cases of ``quiescent'' solitons in Corollary \ref{NLS new soliton} and
Remark \ref{NLS new soliton complex}.

For simplicity of exposition, in the present and following sections we
will choose the value of the integrated covariance of the process $\nu$
to be $\alpha=1/2$.

\subsection{Small-intensity real white noise} \label{subsec NLS W real}

In this subsection we consider the example of an initial condition
composed of a square function perturbed with a small real white noise.
First, we use a perturbative approach to study the effects of the
perturbation on ``true'' solitons (Proposition~\ref{prop NLS 1}). Then,
in the last part of this subsection, we study the effects of this
perturbation on a special structure called ``quiescent'' soliton
(Corollary \ref{NLS new soliton}).

We have here $U_0(x) = (q + \sigma\dot W_x ) \mathbf{1}_{[0,R]}(x)$.
The initial condition is (\ref{NLS - CI}) and the system (\ref{ZSSP})
for $x\in[0,R]$ reads
%
%
\begin{equation}
\label{ZSSP - pert} \cases{ \dd\psi_1 = i (q \psi_2 -
\zeta\psi_1 ) \,\dd x + i \sigma\psi_2 \circ\dd
W_x, \vspace*{2pt}
\cr
\dd\psi_2 = i (q \psi_1
+\zeta\psi_2 ) \,\dd x + i \sigma\psi_1 \circ\dd
W_x.}
\end{equation}

%
\begin{prop}\label{prop NLS 1}
For $qR>\frac{\pi}{2}$ and in the limit of a small, real, white
noise-type stochastic perturbation of the initial condition, the
parameter $\eta$ defining the amplitude of the soliton component of the
solution is perturbed at first order by a~small, zero-mean, Gaussian
random variable, which is given by
%
%
\begin{equation}
\label{eq pert 1 NLS} \sigma\frac{ q \sin(c_0 R) W_R + \int_0^R 2
(\eta_0q/c_0) \sin(c_0(R-y) ) \sin(c_0 y) \,\dd W_y}{2 [
q^2/c_0^2 + R \eta_0 ] \sin( c_0 R ) - R(\eta_0^2/c_0)
\cos(c_0 R ) },
\end{equation}
where $c_0:= \sqrt{q^2-\eta_0^2}$, and $\eta_0$ is the parameter
defining the amplitude of the soliton of the unperturbed system.

The velocity of the soliton remains unchanged.
\end{prop}

%
\begin{oss}\label{oss:IC real-complex}
With the same proof, one can show that this result also holds for a
purely imaginary (deterministic) initial condition perturbed by a
purely imaginary, small white noise. And a simple phase shift
\cite{Ki89} allows to extend this result to any complex initial
condition $U_0(x)= q \mathbf{1}_{[0,R]}(x)$, $q\in\mathbb{C}$ perturbed
with a white noise with the same constant phase of $q$.
\end{oss}

%
\begin{oss}
With the same proof, it is possible to show that solitons are stable
with respect to small random perturbations under more general
hypothesis. In particular,\vadjust{\goodbreak} the perturbation need not be rapidly
oscillating. For example, substituting the white noise with a general
process $Q_x$ in Corollary \ref{NLS new soliton} below one obtains the
same result: $\partial_\sigma\eta= q \int_0^R Q_x \,\dd x$, so that a
true soliton is created whenever the integral is positive.

Indeed, another possible and equivalent approach for rapidly
oscillating processes would be to work with the original process
$\nu_\varepsilon$ for the IST and carry out the scaling limit only at
this stage.
\end{oss}

\begin{pf*}{Proof of Proposition \ref{prop NLS 1}}
Proposition \ref {prob reg zeta-sigma} ensures that equation (\ref{ZSSP
- pert}) with initial condition (\ref{NLS - CI}) defines a stochastic
flow $\Psi ^{(\zeta,\sigma)}(x)$ of $\mathcal{C}^1$-diffeomor\-phisms,
which is $\mathcal{C}^1$ also in the parameters $(\xi, \eta, \sigma)$.
Looking at the flow at point $R$ we can define a complex-valued
function of $\Psi(R)$ as $F(\xi,\eta,\sigma):= \psi_1^{(\zeta,
\sigma)}(R)$. We look for the set of values of $(\xi,\eta,\sigma)$
corresponding to zeros of the function $F$: they are the parameters
$(\zeta=\xi+i\eta)$ defining the soliton components of the solution of
the problem perturbed with a noise of amplitude $\sigma$. We claim that
a small stochastic perturbation has only the effect of a small
variation in the value of $\zeta=\xi+i\eta$ with respect to the value
$\zeta_0=i\eta_0$ of the corresponding soliton in the deterministic
case. We will prove this using the implicit function theorem: for any
fixed and sufficiently small $\sigma$, $F(\xi,\eta,\sigma)$ has a
unique zero in some open set containing the point $(0,\eta_0)$.

Since Lemma \ref{lemm J zeta-sigma} in the \hyperref[sec:app]{Appendix}
guarantees that the Jacobian matrix $J$ of the derivatives of $F$ with
respect to $\xi$ and $\eta$ is invertible, we can apply the implicit
function theorem at point $(\xi,\eta,\sigma)=(0,\eta_0,0)$. Fix
$\zeta=i\eta_0$. By Proposition \ref{prob reg zeta-sigma} the flow
defined by the system (\ref{ZSSP - pert}) is $\mathcal{C}^1$ in the
parameters, so that its derivative in~$\sigma$ coincides with the first
term of the Taylor expansion, denoted $\Psi^{(1)}$,
%
%
\begin{equation}
\label{psi NLS ord 1} \dd\Psi^{(1)} = \pmatrix{ \eta_0 & iq
\cr
iq & -\eta_0 } \Psi^{(1)} \,\dd x + i \pmatrix{ 0 & 1
\cr
1 & 0 } \Psi^{(0)} \,\dd W_x.
\end{equation}
Here, $\Psi^{(0)}$ denotes the solution of the deterministic problem
($\sigma=0$). Let $M$ be the matrix appearing in the drift term of the
above equation; the solution can be computed explicitly,
\begin{eqnarray*}
\Psi^{(1)} (x)&=& i \int_0^x \exp
\bigl(M(x-y) \bigr) \pmatrix{ 0 & 1
\cr
1 & 0 } \Psi^{(0)} (y) \,\dd
W_y,
\end{eqnarray*}
where
\begin{eqnarray*}
\bigl[\exp \bigl(M(x-y) \bigr) \Psi^{(0)} (y) \bigr]_1 &=&
i \frac{q}{c_0}\cos \bigl(c_0(x-y) \bigr) \sin(c_0
y)
\\
&&{} + 2 i \frac{\eta_0q}{c_0^2} \sin \bigl(c_0(x-y) \bigr)
\sin(c_0 y)
\\
&&{}+ i\frac{q}{c_0}\sin \bigl(c_0(x-y) \bigr)
\cos(c_0 y)
\\
&=& i\frac{q}{c_0}\sin(c_0 x) + 2 i \frac{\eta_0q}{c_0^2} \sin
\bigl(c_0(x-y) \bigr) \sin(c_0 y).
\end{eqnarray*}
It follows that
\[
\partial_\sigma F(0,\eta_0,0) = i \int_0^R
\bigl[\exp \bigl(M(R-y) \bigr) \Psi^{(0)}_y
\bigr]_1 \,\dd W_y.
\]
In the proof of Lemma \ref{lemm J zeta-sigma} the derivatives
$\partial_\xi F(0,\eta_0,0)$ and $\partial_\eta F(0,\eta_0,0)$ are
computed explicitly: the first one is imaginary pure, while the latter
is real. Set $\alpha:= \partial_\eta F (0,\eta_0,0) = [
\frac{q^2}{c_0^3} + R \frac{\eta_0}{c_0} ] \sin( c_0 R ) -
R\frac{\eta_0^2}{c_0^2} \cos(c_0 R )$.\vadjust{\goodbreak} Using the explicit formula for
the Jacobian obtained in Lemma \ref{lemm J zeta-sigma} we can write
\[
J^{-1}= \pmatrix{ 0 & -\dfrac{1}{\alpha}
\vspace*{3pt}\cr
\dfrac{1}{\alpha} & 0 }
\]
and from the formula
\[
v:=\pmatrix{ \partial_\sigma\xi
\cr
\partial_\sigma\eta} (
\sigma=0) = - J^{-1} \pmatrix{ \Re( \partial_\sigma F )
\cr
\Im( \partial_\sigma F ) } (0, \eta_0,0)
\]
we get that $\partial_\sigma\eta$ is given by (\ref{eq pert 1 NLS})
and that $\partial_\sigma\xi= 0$.
\end{pf*}

When $q=(2n+1)\pi/2R$ for some $n\in\mathbb{N}$, in the
deterministic case we
have the creation of what is sometimes called a ``quiescent'' soliton.
Also in this case, the stochastic perturbation can modify, at first
order, only the amplitude of the soliton.

%
\begin{cor}\label{NLS new soliton}
When the background deterministic solution contains a ``quiescent''
soliton, the stochastic perturbation destroys it with probability $1/2$
and transforms it into a true soliton (with a positive amplitude) with
probability~$1/2$.
\end{cor}

\begin{pf} In the case of a ``quiescent'' soliton Lemma \ref{lemm J
zeta-sigma} still holds. Indeed we have now $c_0=q$ and $qR=(2n+1)\pi
/2$, so that equation (\ref{NLS pert cond 1}) in the proof of the
lemma becomes
\[
\partial_\xi F (0,0,0) = -i \biggl[ \frac{q^2}{c_0^3} +
R \frac{\zeta_0}{c_0} \biggr] \sin( c_0 R ) +i R\frac{\zeta_0^2}{c_0^2}
\cos(c_0 R ) = \mp i \frac{1}{q} \neq0
\]
and the determinant of the Jacobian is not zero. One has then
$\alpha=\pm1/q$ and
\begin{eqnarray*}
\partial_\sigma F(0,0,0) &=& - \int_0^R
\sin( qR ) \,\dd W_y = \mp W_R,
\end{eqnarray*}
which is real. Therefore,
\begin{eqnarray*}
v:=\pmatrix{ \partial_\sigma\xi
\cr
\partial_\sigma\eta} (
\sigma=0) = - \pmatrix{ 0 & \mp q
\cr
\pm q & 0 } \pmatrix{ \mp W_R
\cr
0 } = \pmatrix{ 0
\cr
q W_R }.
\end{eqnarray*}
A true soliton is created whenever $W_R>0$.
\end{pf}

\subsection{Small-intensity complex white noise} \label{subsec NLS W complex}

The limit case of a small-ampli\-tude complex white noise is similar to
the case of the real white noise treated above. Take $U_0(x)= (q+\sigma
\dot W_x) \mathbf{1}_{[0,R]}(x)$ where $W_x= W^{(1)}_x + i W^{(2)}_x$
is a complex Wiener process. We have the following proposition.

%
\begin{prop}\label{prop NLS complex}
For $qR>\frac{\pi}{2}$ and in the limit of a small complex white
noise-type stochastic perturbation of the initial condition, the
parameters $\xi, \eta$ defining the velocity and amplitude of each
soliton are perturbed at first order by small, zero-mean, Gaussian
random variables, which are given by
\begin{eqnarray*}
\label{eq pert 2 NLS} \partial_\sigma\xi&=& - \frac{ q \sin(c_0 R)
W_R^{(2)} + \int_0^R 2 (\eta_0q/c_0^2) \sin(c_0(R-y) ) \sin
(c_0 y) \,\dd W_y^{(2)}}{2 [ q^2/c_0^2 + R \eta_0 ] \sin(c_0 R
) - R(\eta_0^2/c_0) \cos(c_0 R ) },
\\
\partial_\sigma\eta&=& \frac{q \sin(c_0 R) W_R^{(1)} + \int_0^R 2
(\eta_0q/c_0^2) \sin(c_0(R-y) ) \sin(c_0 y) \,\dd W_y^{(1)}}{2
[q^2/c_0^2 + R \eta_0 ] \sin( c_0 R ) - R(\eta
_0^2/c_0) \cos(c_0 R ) },
\end{eqnarray*}
where $c_0:= \sqrt{q^2-\eta_0^2}$, and $\eta_0$ is the parameter
defining the amplitude of the soliton of the unperturbed
system.\vspace*{-1pt}
\end{prop}

\begin{pf} This proof is similar to that of Proposition \ref{prop NLS
1}. An analogy of Proposition \ref{prob reg zeta-sigma} holds in this
setting, and Lemma \ref{lemm J zeta-sigma} remains unchanged (note
that in Lemma \ref{lemm J zeta-sigma} we work on the deterministic
equation). From equation (\ref{psi NLS ord 1}) onward one has just to
remember that $W$ is now complex.\vspace*{-1pt}
\end{pf}

%
\begin{oss}
In this case, since we used a noise with a symmetric law, the first
order perturbations of the velocity and amplitude of the soliton have
the same law and are independent. We could have taken a nonsymmetric
complex noise to perturb the initial condition: $\widetilde\nu= \nu_1
+ i\nu_2$, where $\nu_1$ and $\nu_2$ have different distributions. In
this case, the perturbations of the velocity and amplitude of the
soliton would still be independent, but not sharing the same
law.\vspace*{-1pt}
\end{oss}

%
\begin{oss}\label{NLS new soliton complex}
``Quiescent'' solitons are perturbed in both amplitude and speed,
leading to the possible creation of true solitons with nonzero
velocity. The same result holds when perturbing the initial data with
more general complex processes.\vspace*{-1pt}
\end{oss}

\section{Stability of KdV solitons}\label{sec KdV}

In this section we study our second example, the KdV equation. We focus
on the soliton components of the solution. Solitons of the KdV equation
are identified by an imaginary number $\zeta=i\eta$, defining both the
velocity and amplitude of the soliton, which are related. Recall that
in Section~\ref{subsec KdV det} we have presented some classical
results on the background deterministic solution. Using these results
in Section~\ref{subsec KdV W} we analyze how this solution is modified
by the introduction of a small-amplitude white noise perturbation of
the initial condition: the main result is contained in Proposition
\ref{prop KdV 1}, while Proposition \ref{KdV new soliton} deals with
the case of ``quiescent'' solitons. The last subsection deals with the
case of a perturbation of the zero initial condition.

\subsection{Small-amplitude random perturbation with $q>0$} \label{subsec KdV W}


Let the initial condition of the KdV equation be given by\vadjust{\goodbreak} $U_0= (q +
\sigma\dot W_{x} ) \mathbf{1}_{[0,R]}(x) $, where $W$ is a~standard
Wiener process. Since solitons correspond to zeros of the complex
extension of $a$, which in turn must be located on the imaginary axis,
we look for bounded solutions of equation (\ref{KdV lax 1}) for
$\zeta=i\eta, \eta\in\mathbb{R}^+$. The first main result is contained
in the following proposition.

%
\begin{prop}\label{prop KdV 1}
In the limit of a small, white noise-type stochastic perturbation of
the initial condition, the parameter $\eta$ defining the velocity and
amplitude of the generated soliton is perturbed at first order by a
small, zero-mean, Gaussian random variable, which is given by
%
%
\begin{equation}\label{eq ratio W-eta}
\qquad\frac{\sigma\int_0^R \varphi_0(R-x)
\varphi_0(x) \,\dd W_x} {\cos(c_0 R) [ 2+\eta_0 R - \eta_0^3
R/c_0^2 ] + \sin(c_0 R) [ (3\eta_0 + 2\eta_0^2 R)/c_0 +
\eta_0^3/c_0^3 ]}.
\end{equation}
Here, $\varphi_0$ is the deterministic solution of (\ref{varphi_xx}),
given by (\ref{varphi0}), $c_0:= \sqrt{q^2-\eta_0^2}$, and $\eta_0$ is
the parameter corresponding to the unperturbed soliton.
\end{prop}

Before we start the proof of the above proposition, we shall remark
that an analog of Proposition \ref{prob reg zeta-sigma} holds in this
setting, and it provides the uniqueness and regularity results for the
solution. The result reads in the following way:

%
\begin{prop}\label{prob reg eta-sigma}
The stochastic differential equation (\ref{phi-varphi}) below defines a
stochastic flow $\Phi^{(\eta,\sigma)}(x) = (\varphi(x),
\widetilde\varphi(x) )^T$ of $\mathcal{C}^1$-diffeomorphisms, which is
$\mathcal{C}^1$ also in the parameters $\eta,\sigma$.
\end{prop}

\begin{pf*}{Proof of Proposition \ref{prop KdV 1}}
We need to solve
%
%
\begin{equation}
\label{phi-varphi} \cases{ \dd\varphi=
\widetilde\varphi\,\dd x, \vspace*{2pt} \cr \dd\widetilde\varphi=
\bigl(-q + \eta^{2} \bigr) \varphi\,\dd x + \sigma \varphi\circ\dd
W_{x}}\vadjust{\goodbreak}
\end{equation}
for $x\in[0,R]$ with initial conditions $\varphi(0) = 1$, $\widetilde
\varphi(0)=\eta$. As we did for the NLS equation, we look at the flow
provided by Proposition \ref{prob reg eta-sigma} and define a function
of the flow at the point $x=R$ as $F(\eta,\sigma):=
\widetilde\varphi(R) + \eta\varphi(R)$: this remains a function of the
two parameters $(\eta,\sigma)$. To prove Proposition \ref{prop KdV 1}
we use the implicit function theorem to show that $F(\eta,\sigma)$ has
a unique zero in some open set containing $(\eta_{0},0)$, the point
corresponding to the deterministic solution. Since Lemma \ref{lem J
KdV} in the \hyperref[sec:app]{Appendix} guarantees that
$\partial_{\eta} F (\eta_0,0)\neq0$, we can apply the implicit function
theorem. We have (at $\eta_0$)
\[
\cases{ \dd\partial_\sigma\varphi= \partial_\sigma\widetilde
\varphi\,\dd x, \vspace*{2pt}
\cr
\dd\partial_\sigma \widetilde\varphi=
\bigl(-q+\eta_0^2 \bigr) \partial_\sigma\varphi\,
\dd x - (\varphi+ \sigma\partial_\sigma \varphi) \circ\dd
W_x.}
\]
By Proposition \ref{prob reg eta-sigma} the flow
$\Phi_x^{(\eta,\sigma)}$ is $\mathcal{C}^1$ in the parameters, so
that its
derivative in $\sigma$ at $\sigma=0$ coincides with the first term of
the Taylor expansion, which is the solution of
\[
\cases{ \dd\partial_\sigma\varphi= \partial_\sigma\widetilde
\varphi\,\dd x, \vspace*{2pt}
\cr
\dd\partial_\sigma \widetilde\varphi=
\bigl(-q+\eta_0^2 \bigr) \partial_\sigma\varphi\,
\dd x - \varphi_0 \,\dd W_x.}
\]
In matrix notation,
\[
\dd\pmatrix{ \partial_\sigma\varphi
\cr
\partial_\sigma
\widetilde\varphi} = \dd\Phi= \lleft[\matrix{ 0 & 1
\cr
-c^{2} &
0} \rright] \Phi\,\dd x - \pmatrix{ 0
\cr
\varphi_0 \,\dd
W_x }
\]
with
\begin{eqnarray*}
M&=&\lleft[\matrix{ 0 & 1 \cr -c^2 & 0} \rright],
\\
e^{Mx} &=& \lleft[\matrix{ \cos(c x) & \dfrac{1}{c} \sin(c x) \vspace*{4pt}\cr - c
\sin(c x) & \cos(c x)} \rright].
\end{eqnarray*}
The solution is
\[
\Phi(x) = \Phi(0) - \int_0^x e^{M (x-y)}
\pmatrix{ 0
\cr
\varphi_0 (y) } \,\dd W_y.
\]
We have
\begin{eqnarray*}
\partial_\sigma\varphi(x) &=& - \frac{1}{c} \int
_0^x \sin \bigl(c(x-y) \bigr) \biggl[ \cos(c y)
+ \frac{\eta}{c} \sin(c y) \biggr] \,\dd W_y,
\\
\partial_\sigma\widetilde\varphi(x) &=& - \int_0^x
\cos \bigl(c(x-y) \bigr) \biggl[ \cos(c y) + \frac{\eta}{c} \sin(c y) \biggr]
\,\dd W_y.
\end{eqnarray*}
Therefore,
\[
\partial_\sigma F (\eta_0,0) = \partial_\sigma
\widetilde\varphi(R) + \eta_0 \partial_\sigma\varphi(R) = -
\int_0^R \varphi_0(R-x)
\varphi_0 (x) \,\dd W_x
\]
and we obtain, at first order, the perturbation of the value of the
parameter $\eta$ defining the velocity and amplitude of the soliton:
for every $(q,R,\eta_{0})$,
%
%
\begin{eqnarray*}
\partial_\sigma\eta(\sigma) &=& -\frac{\partial_\sigma F}{\partial
_\eta F} (
\eta_0,0)
\\
&=& \biggl(\int_0^R \biggl[\cos \bigl(c_0(R - x) \bigr)+ \frac{\eta_0}{c_0}\sin\bigl(c_0(R
- x) \bigr) \biggr]
\\
&&\hspace*{46pt}{}\times \biggl[ \cos(c_0 x) + \frac{\eta_0}{c_0}\sin(c_0 x) \biggr] \,\dd W_x\biggr)
\\
&&{} \bigg/ \biggl(\cos(c_0 R) \biggl[ 2+\eta_0 R - \frac{\eta _0^3 R}{c_0^2} \biggr] + \sin(c_0 R)
\biggl[\frac{3\eta_0 + 2\eta_0^2 R}{c_0} + \frac{\eta_0^3}{c_0^3}\biggr]\biggr).
\end{eqnarray*}
The proposition is proved.
\end{pf*}

We turn now to consider the case when the stochastic perturbation can
result in the creation of a new soliton. As for the NLS equation, this
happens only for specific ``critical'' values of $q$ (and $R$). The
result is presented in the next proposition.

%
\begin{prop}\label{KdV new soliton}
If we are at a critical point, which is to say $\sqrt{q} R = n \pi$ for
$n\in\mathbb{N}$, $n>0$, a small-amplitude white noise-type stochastic
perturbation of the potential \textup{may} create a new small-amplitude
soliton. The condition for the creation of a new soliton is that the
zero-mean, Gaussian random variable (\ref{eq ratio W-eta 3}) is
positive.
\end{prop}

\begin{pf} For a generic $\eta_0$, the deterministic background
solution is ($c_0=\sqrt{q-\eta^2_0}$)
%
%
\begin{equation}
\label{det sol KdV} \varphi_0 = \cos(c_0 x) +
\frac{\eta_0}{c_0} \sin(c_0 x).
\end{equation}
We have $U_0=q + \sigma\dot W_x$; we want to apply the implicit
function theorem to the function $F(\eta,\sigma)$ defined above, at the
point $(0,0)$. At this point
\begin{eqnarray*}
\partial_{\eta}\varphi(R) &=& \frac{1}{\sqrt{q}} \sin(\sqrt{q} R) =0,
\\
\partial_\eta\widetilde\varphi(R) &=& \cos(\sqrt{q} R) = \pm1.
\end{eqnarray*}
Therefore
\begin{eqnarray*}
\partial_{\eta} F (0,0) &=& \partial_{\eta} \widetilde\varphi+
\varphi+ \eta_0 \partial_{\eta} \varphi= 2 \cos(\sqrt{q} R)
= \pm2
\end{eqnarray*}
and the implicit function theorem is applicable. We also have (again at
$\eta_0=0$)
\[
\cases{\dd\partial_\sigma
\varphi= \partial_\sigma\widetilde\varphi\,\dd x, \vspace*{2pt}
\cr
\dd
\partial_\sigma\widetilde\varphi= -q \partial_\sigma\varphi\,\dd
x - (\varphi+ \sigma\partial_\sigma\varphi) \circ\dd W_x.}
\]
As seen above, the solution of the above system coincides at $\sigma=0$
with the solution of
\[
\cases{ \dd\partial_\sigma\varphi= \partial_\sigma\widetilde
\varphi\,\dd x, \vspace*{2pt}
\cr
\dd\partial_\sigma\widetilde\varphi= -q
\partial_\sigma\varphi\,\dd x - \varphi_0 \,\dd
W_x,}
\]
which is given by (here, $c=\sqrt{q}$)
\begin{eqnarray*}
\partial_\sigma\varphi(x) &=& - \frac{1}{\sqrt{q}} \int
_0^x \sin \bigl(\sqrt{q}(x-y) \bigr) \cos(
\sqrt{q} y) \,\dd W_y,
\\
\partial_\sigma\widetilde\varphi(x) &=& - \int_0^x
\cos \bigl(\sqrt{q}(x-y) \bigr) \cos(\sqrt{q} y) \,\dd W_y.
\end{eqnarray*}
We have therefore
\begin{eqnarray*}
\partial_\sigma F (0,0) &=& - \int_0^R
\cos \bigl(\sqrt{q} (R-x) \bigr) \cos(\sqrt{q} x) \,\dd W_x
\end{eqnarray*}
and since $\cos(\sqrt{q} R)=\pm1$
%
%
\begin{eqnarray}
\label{eq ratio W-eta 3} v&:=& \partial_\sigma\eta(\sigma) = -
\frac{\partial_\sigma
F}{\partial_\eta F} (0,0) = \frac{\int_0^R \cos(\sqrt{q} (R-y) ) \cos(\sqrt{q} y) \,\dd W_y} {2
\cos(\sqrt{q} R)}
\nonumber\\[-8pt]\\[-8pt]
&=& \frac{1}{2} \int_0^R
\cos^2(\sqrt{q}x) \,\dd W_x.\nonumber
\end{eqnarray}
In this case, a new small-amplitude soliton, corresponding to $\eta=
\sigma v$, is created whenever $v>0$.
\end{pf}

\subsection{Small-amplitude random perturbations with $q=0$}

We analyze now the case in which the initial condition is the pure
stochastic perturbation: contrarily to the NLS equation, even this
small initial condition can generate a soliton. In this setting there
is no (nontrivial) solution in the deterministic case. This case is
obtained at the critical point $\sqrt{q} R=0$, where the first
quiescent soliton is created.

%
\begin{prop}\label{prop KdV 2}
For $q=0$, a small-amplitude stochastic perturbation of the potential
\textup{may} create a new small-amplitude soliton. For a white noise-type
perturbation $\dot W_x$, a new soliton is created if $W_R>0$ (which is
an event of probability $1/2$) and the created soliton corresponds to
the value $\eta=\sigma W_{R}$.
\end{prop}

\begin{pf} Again, we want to use the implicit function theorem; take
the limit of the deterministic solution (\ref{det sol KdV}) as $q\to0$,
\[
\varphi_0=1, \qquad\qquad\widetilde\varphi_0 =0.
\]
We have
\[
\partial_{\eta} F (0,0) = \partial_{\eta} \widetilde\varphi+
\varphi+ \eta_0 \partial_{\eta} \varphi= 1
\]
and ($\eta_0=0$)
\begin{eqnarray*}
&\cases{ \dd\varphi= \widetilde\varphi\,\dd x, \vspace*{2pt}
\cr
\dd\widetilde
\varphi= -\sigma\varphi\circ\dd W_{x},}&
\\
&\cases{ \dd\partial_\sigma\varphi= \partial_\sigma\widetilde
\varphi\,\dd x, \vspace*{2pt}
\cr
\dd\partial_\sigma\widetilde\varphi= - (
\varphi+ \sigma\partial_\sigma\varphi) \circ\dd W_x.}&
\end{eqnarray*}
Therefore,
\[
\partial_\sigma F (0,0) = \partial_\sigma\widetilde\varphi+
\eta_0 \partial_\sigma\varphi= - \int_0^R
\varphi\,\dd W_x = -W_R.
\]
It follows that
\[
\partial_\sigma\eta(\sigma) = -\frac{\partial_\sigma F }{\partial
_\eta F } (0,0) =
W_R
\]
and a single soliton corresponding to $\eta=\sigma W_R$ is generated
whenever $W_R>0$, which means with probability $\frac{1}{2}$.\vadjust{\goodbreak}
\end{pf}

%
\begin{oss}
The same result of Proposition \ref{prop KdV 2} holds with more general
processes: if, instead of a white noise, we take $U_0 (x)=\sigma
Q_x$ in (\ref{varphi_xx}), where $Q_x$ is a generic stochastic process,
we have (at $\eta=0$)
\begin{eqnarray*}
&\cases{ \dd\partial_\sigma\varphi= \partial_\sigma\widetilde
\varphi\,\dd x, \vspace*{2pt}
\cr
\dd\partial_\sigma\widetilde\varphi= -
Q_x (\varphi+ \sigma\partial_\sigma\varphi) \,\dd x,}&
\\
& \displaystyle\partial_\sigma F (0,0) = \partial_\sigma\widetilde\varphi+
\eta_0 \partial_\sigma\varphi= - \int_0^R
Q_x \,\dd x&
\end{eqnarray*}
and
\[
\label{eq q} \partial_\sigma\eta(\sigma) = -\frac{\partial
_\sigma F }{\partial_\eta F } (0,0) =
\int_0^R Q_x \,\dd x.
\]
In this case, a single soliton is created whenever the noise introduced
has positive mean, and it corresponds to $\eta=\sigma\int_0^R Q_x \,\dd
x $.
\end{oss}

%
\begin{oss} The mass of a soliton of the KdV equation is $M_\eta=\break \int_\mathbb{R} U(x) \,\dd x = 4\eta$. Here, we have introduced a
perturbation of mass $M_{U_0} = \sigma\int_\mathbb{R} Q_x \,\dd x =
\eta$. We see therefore that the soliton created has a larger mass than
the initial perturbation, implying that the radiative part (going in
the direction opposite to that of the soliton) has absorbed a total
mass of $3\eta$.

The energy conversion efficiency from a small noise source to a soliton
is very poor, since the input energy is of order $\sigma^2$ ($E_{U_0}=
\sigma^2 \int_\mathbb{R} Q_x^2 \,\dd x$), while the energy of the
created soliton is only $E_\eta= \int_\mathbb{R} U^2(x) \,\dd x =
\frac{16}{3}\eta^3\sim\sigma^3$, for a smooth source. For a white noise
source, the input energy is infinite, while the energy of the created
soliton is finite.
\end{oss}
%

\section{Conclusion and comments}

For the NLS and KdV equations, the study of soliton emergence from a
localized, bounded initial condition perturbed by a~wide class of
rapidly oscillating random processes can be reduced to the study of
a~canonical system of SDEs, formally corresponding to the white noise
perturbation of the initial condition. The integrated covariance is the
only parameter of the perturbation process that influences the limit
system of SDEs. From the study of this limit system, one obtains
quantitative information on the modification of solitons due to the
random perturbation.

For the NLS equation we have a threshold effect: if for the
deterministic initial condition the integral $\int_\mathbb{R} U_0(x)
\,\dd x$ exceeds $\pi/2$, at least one soliton is created. In this
case, a small-amplitude random perturbation of the initial condition
results in a small variation in amplitude for the created soliton.
However, if the phase of the (possibly complex) perturbation is
constant and equal to the phase of the deterministic background initial
condition, the speed of the created soliton is not modified. On the
contrary, a complex perturbation with varying phase can modify both the
amplitudes and speeds of solitons.\vadjust{\goodbreak}

Since the KdV equation does not present a threshold phenomenon, a
small-amplitude random perturbation always results in a small variation in
both speed and amplitude of the created soliton, for any nonnegative
initial condition.

As for a ``quiescent'' soliton, for both NLS and KdV a stochastic
perturbation has a positive probability (depending on the type of
perturbation used) of creating a new real soliton.

The results of Section~\ref{section W} on the canonical system of SDEs
for rapidly oscillating processes holds for any value $\sigma$ of the
amplitude of the perturbation, and the results on the stability of
solitons with respect to small random perturbations ($\sigma\ll1$) hold
also without the assumption of a rapidly oscillating process. However,
a general framework to treat the problem of creation of solitons
without the assumptions of a rapidly oscillating initial condition or
smallness of the random perturbation seems not to be available at the
moment. In particular, the result strongly depends on the kind of
random process used in the initial condition. However, specific cases
can be treated with ad-hoc techniques: some examples will be provided
in a subsequent publication.

\begin{appendix}\label{sec:app}
\section*{Appendix: Technical lemmas}
We collect here a few technical results needed for the proof of Theorem
\ref{teo convergenza Phi e Psi} and Propositions \ref{prop NLS 1} and
\ref{prop KdV 1}. We shall retain the different notation introduced
there. 
We start by showing the proof of Lemma \ref{lemma approx H}.

\begin{pf*}{Proof of Lemma \ref{lemma approx H}}
A subset $\mathcal{A}\subset\mathcal{H}$ of the Hilbert space
$\mathcal{H}$ is
rela\-tively compact if and only if it is bounded, and for every
$\lambda>0$ there exists a finite-dimensional subspace $\mathcal
{H}_n\subset
\mathcal{H}$ s.t.
\[
\sup_{h\in\mathcal{A}} d_\mathcal{H}( h, \mathcal{H}_n
) < \lambda.
\]
Therefore, the two conditions (\ref{approx H}) are necessary to have
that for any $\kappa>0$ there exists a compact subset $\mathcal
{A}_\kappa \subset\mathcal{H}$ s.t.
\[
\sup_{\varepsilon\in(0,1]} \mathbb{P} \bigl( X^\varepsilon\in\mathcal{H}
\setminus\mathcal{A}_\kappa \bigr) \le\kappa.
\]
The two conditions (\ref{approx H}) are also sufficient. Indeed, if
they are satisfied for any given $\kappa>0$ one obtains a compact
subset $\mathcal{A}_\kappa$ of $\mathcal{H}$ considering the closure of
\[
\mathcal{B}_\kappa= \mathcal{H}\setminus\bigcup
_{n\ge1} \bigl( \bigl\{ h | \|h\| > \rho_\kappa\bigr\} \cup \bigl
\{ h | d_\mathcal{H} ( h, \mathcal{H}_{\kappa/2^n, 1/n} ) > 1/n \bigr\}
\bigr).
\]
Then one obtains that
\[
\mathbb{P} \bigl( X^\varepsilon(x) \in\mathcal{H}\setminus
\mathcal{A}_\kappa \bigr) \le2 \kappa.
\]
The lemma is proved.
\end{pf*}

%
\begin{lemma}\label{lemma:construc-Hn}
We construct an explicit example of the subspaces
$\mathcal{H}_n\subset\mathcal{H}=W^{3,2}(G)$ to be used in the proof
of Theorem
\ref{teo convergenza Phi e Psi} and Lemma \ref{lemma tecnico approx
Hn}.
\end{lemma}

\begin{pf*}{Construction of $\mathcal{H}_n$}
Divide $G=(-N,N)^3$ into cubes with sides of length $1/n$ and add one
extra layer of cubes around it:
\[
A_{i,j,k}:= \bigl[ i/n, (i+1)/n \bigr)\times\bigl[ j/n, (j+1)/n \bigr)\times\bigl[ k/n, (k+1)/n \bigr)
\]
for $i,j,k=-(nN+1),\ldots, nN$. Define the piecewise (on every cube)
polynomials of fourth degree as
%
%
\begin{equation}
\widetilde h(x):= \sum_{i,j,k=-(Nn+1)}^{Nn}
\sum_{m=0}^4 \frac{1}{m!} \bigl
\langle a^{(m)}_{i,j,k} | x-y_{i,j,k} \bigr
\rangle^{(m)} \mathbh{1}_{ \{ {A_{i,j,k}} \}} (x), \label{eq h tilde}
\end{equation}
where\vspace*{-3pt} $y_{i,j,k}$ is the center of the cube $A_{i,j,k}$,
$a^{(m)}_{i,j,k}$ are families of $m$-dimensional tensors and the
brackets denote the\vspace*{-1pt} relative tensor products [so that, e.g., $\langle
a^{(4)}_{i,j,k} | x - y_{i,j,k} \rangle^{(4)}$ denotes the product
between the four-dimensional tensor $a^{(4)}_{i,j,k}$ and four copies
of the vector $(x-y_{i,j,k}) $]. With these definitions $\widetilde h$
is a function defined on $[-N-\frac{1}{n},N+\frac{1}{n})^3$, but its
restriction to $G$ does not belong to $\mathcal{H}$ in general since it
may not even be continuous. Let $\Gamma$ be a real, nonnegative, smooth
function, with compact support contained in $[-1/2,1/2]^3$ and such
that $\int_{[-1/2,1/2]^3} \Gamma(y) \,\dd y =1$. Setting $\Gamma^n(y):=
n^3 \Gamma(ny)$ we can finally define $\mathcal{H}_n$ as the
finite-dimensional space of functions of the form $h(x):=(\widetilde h
\star\Gamma^n) (x)$. Remark that $\widetilde h$ has been defined on a
set larger than $G$, so that the convolution product is well defined
for $x\in G$.
\end{pf*}

%
\begin{lemma}\label{lemma tecnico approx Hn}
For every $g\in\mathcal{C}^4 (G)$, there exists a $g_n\in\mathcal
{H}_n$ s.t.
\[
\|g-g_n \|_{\mathcal{H}} \le C \frac{1}{n} \| g
\|_{\mathcal{C}^4(G)}.
\]
\end{lemma}

\begin{pf}
\textit{Step} 1 (Construction of $g_n$). For $i,j,k = -nN,\ldots, nN-1$
(which means that $y_{i,j,k} \in G$) set $a_{i,j,k}^{(m)}:= D^{m} g
(y_{i,j,k}) $. For\vspace*{-1pt} $i,j,k$ such that $y_{i,j,k} \notin G$ set
$a_{i,j,k}^{(m)}:= D^{m} g (y') $, where $y'$ is the\vspace*{-1pt} nearest cube
center; notice that the distance of these two points is at most the
diameter of the cubes, which we call $2\delta:= 2\sqrt{3} n^{-1}$. With
the $a_{i,j,k}^{(m)}$ thus defined\vspace*{1pt} we construct the piecewise
polynomial function $\widetilde g_n$ as in (\ref{eq h tilde}): on the
cubes the centers of which are not in $G$, this function is just a copy
of the function defined on the nearest cube with center in $G$.
Finally, we define $g_n$ as the convolution product $g_n:= \Gamma^n
\star\widetilde g_n$.

\textit{Step} 2 (Estimates). For every multiindex $a\in\mathbb{N}^3$
such that $|a|_{1}:= a_1 + a_2 +\break a_3 \le3$, we need to estimate
\[
\int_{G} \bigl| \partial^a \bigl(
\Gamma^n\star\widetilde g_n \bigr) (x) -
\partial^a g(x) \bigr|^2 \,\dd x.
\]
To clarify the procedure to obtain an estimate for the above term, we
first give explicit computations for the case $a=e_1=(1,0,0)$. Recall
that, by definition,
\[
\sum_{i,j,k}\int_{A_{i,j,,k}}
\Gamma^n (x-y) \,\dd y = \int_{[-(N+1/n),N+1/n)^3}
\Gamma^n(x-y) \,\dd y =1.
\]
For $a=(1,0,0)$, using twice the inequality $(a+b)^2\le2(a^2+b^2)$, we
have
%
%
\begin{eqnarray}
&& \bigl| \partial_1 \bigl(\Gamma^n \star\widetilde g_n \bigr) (x) - \partial_1 g(x) \bigr|^2\nonumber
\\
&&\qquad = \biggl|\sum_{i,j,k} \int_{A_{i,j,k}} - \partial_{y_1} \Gamma^n (x-y) \widetilde g_n (y) \,\dd y - \partial_1 g(x) \biggr|^2 \nonumber
\\
&&\qquad \le 2 \biggl\{ \biggl| \sum_{i,j,k} \int _{A_{i,j,k}} \Gamma^n (x-y) \partial_{y_1} \widetilde g_n (y) \,\dd y - \partial_1 g(x) \biggr|^2\nonumber
\\
&&\hspace*{43pt}{} + \biggl| \sum_{i,j,k} \int_{\partial_1 A_{i,j,k}}
\Gamma^n (x-y) \bigl[ \widetilde g_n \bigl(y_1^+,y_2,y_3 \bigr) - \widetilde g_n \bigl(y_1^-,y_2,y_3 \bigr)\bigr] \,\dd y_2 \,\dd y_3 \biggr|^2 \biggr\}\hspace*{-6pt}
\nonumber\\[-20pt]\\[2pt]
&&\qquad \le 4 \biggl\{ \biggl| \sum_{i,j,k} \int
_{A_{i,j,k}} \Gamma^n (x-y) \bigl| \partial_{y_1}
\widetilde g_n (y) - \partial_1 \widetilde
g_n(x) \bigr| \,\dd y \biggr|^2\nonumber
\\
&&\hspace*{43pt}{} + \bigl| \partial_1 \widetilde g_n(x) - \partial_1 g(x) \bigr|^2\nonumber
\\
&&\hspace*{43pt}{} + \biggl| \sum_{i,j,k} \int_{\partial_1 A_{i,j,k}}
\Gamma^n (x-y) \bigl[ \widetilde g_n\bigl(y_1^+,y_2,y_3 \bigr) - \widetilde g_n \bigl(y_1^-,y_2,y_3 \bigr)
\bigr] \,\dd y_2 \,\dd y_3 \biggr|^2 \biggr\}\hspace*{-6pt}\nonumber
\\
&&\qquad = 4 \{ S_1 + S_2 + S_3 \},\hspace*{-4pt}\nonumber
\end{eqnarray}
where $\partial_1 A_{i,j,k}$ denotes the faces of the cubes orthogonal
to the direction $e_1:=(1,0,0)$. Notice that the number of nonzero
terms in the sums over $i,j,k$ of this proof is limited to 8 because
the support of $\Gamma^n$ can intersect at most 8 cubes. The term $S_1$
can be bounded by the square of
\[
\biggl(\sum_{i,j,k} \bigl\| \Gamma^n (x-
\cdot) \bigr\|_{L^1(A_{i,j,k})} \biggr) \bigl\| \partial_{1} \widetilde
g_n ( \cdot) - \partial_1 \widetilde g_n(x)
\bigr\|_{L^\infty(B_{1/2n}(x))},
\]
where the second term is regarded as a function of $y$ ($x$ is fixed),
and the \mbox{$L^\infty$-}norm is taken on the ball $B_{1/2n}(x)$, which is
the support of $\Gamma^n(x-y)$. The first term above is 1, and to
estimate the second term, the worst case is when $y$ does not belong to
the same cube as $x$: let us say that $y\in A^{(1)}$ and $x\in
A^{(2)}$, where $y^{(1)}$ and $y^{(2)}$ are the centers of the cubes
$A^{(1)}$ and $A^{(2)}$, respectively. We have therefore the bound
\begin{eqnarray*}
&&\bigl\| \partial_{1} \widetilde g_n (\cdot) -
\partial_1 \widetilde g_n(x) \bigr\|_{L^\infty(B_{1/(2n)}(x))}
\\
&&\qquad \le\bigl\| \partial_{1} \widetilde g_n (\cdot) -
\partial_{y_1} \widetilde g_n \bigl(y^{(1)} \bigr)
\bigr\|_{L^\infty} + \bigl| \partial_1 g \bigl(y^{(1)} \bigr) -
\partial_1 g \bigl(y^{(2)} \bigr) \bigr|
\\
&&\quad\qquad{} + \bigl| \partial_{y_1} \widetilde g_n
\bigl(y^{(2)} \bigr) - \partial_{y_1} \widetilde
g_n(x) \bigr|
\\
&&\qquad \le4 \delta\bigl\| D^2 g \bigr\|_{L^\infty(G)}.
\end{eqnarray*}
This provides the bound for the term $S_1$. Similarly, for $S_2$ we
have the bound
\begin{eqnarray*}
\bigl| \partial_1 \widetilde g_n(x) - \partial_1
g(x) \bigr| & \le&\bigl| \partial_1 \widetilde g_n(x) -
\partial_1 g \bigl(y^{(2)} \bigr) \bigr| + \bigl| \partial_1
g \bigl(y^{(2)} \bigr) - \partial_1 g(x) \bigr|
\\
& \le& 2 \delta\bigl\|D^2 g\bigr\|_{L^\infty(G)}.
\end{eqnarray*}
We still need to estimate $S_3$, which contains the boundary terms
deriving from the discontinuities of $\widetilde g_n$ (and, in the
general case, of its derivatives). This term requires more careful
estimates. With $C_\Gamma:= \sup_x \Gamma(x)$ and using the fact that
$\| \Gamma^n (x-\cdot) \|_{L^1(\partial_1 A_{i,j,k})} \le n C_\Gamma$,
we have
\begin{eqnarray*}
&& \sum_{i,j,k} \bigl\| \Gamma^n (x- \cdot)
\bigr\|_{L^1(\partial_1 A_{i,j,k})} \bigl\| \widetilde g_n \bigl(y_1^+,\cdot
\bigr) - \widetilde g_n \bigl(y_1^-,\cdot \bigr)
\bigr\|_{L^\infty(\partial_1 A_{i,j,k} )}
\\
&&\qquad \le n C_\Gamma\sum_{i,j,k} \bigl\{
\bigl\| \widetilde g_n \bigl(y_1^+,\cdot \bigr) -
g(y_1, \cdot) \bigr\|_{L^\infty(\partial_1 A_{i,j,k} ) }
\\
&&\hspace*{74pt}{} + \bigl\| g(y_1,\cdot) - \widetilde g_n
\bigl(y_1^-,\cdot \bigr) \bigr\|_{L^\infty(\partial_1 A_{i,j,k} ) } \bigr\}
\\
&&\qquad \le8 n C_\Gamma\delta^4 \bigl\| D^4 g
\bigr\|_{L^\infty(G)}
\\
&&\qquad = C \delta^3 \bigl\| D^4 g \bigr\|_{L^\infty(G)}.
\end{eqnarray*}
The sums over $i,j,k$ above are meant as sums over the faces
$\partial_1 A_{i,j,k}$ intersecting the support of $\Gamma^n
(x-\cdot)$, which are at most 4. Collecting all these results, we have
the uniform bound
\begin{eqnarray*}
\bigl| \partial_1 \bigl(\Gamma^n\star\widetilde
g_n \bigr) (x) - \partial_1 g(x) \bigr|^2 & \le&
C \delta^2 \bigl( \bigl\|D^2 g \bigr\|^2_{L^\infty(G)}
+ \delta^{4} \bigl\|D^4 g \bigr\|^2_{L^\infty(G)}
\bigr).
\end{eqnarray*}
We proceed in the same way with higher order derivatives to get, for a
generic derivative $a$ of order $0\le|a|\le3$, the estimate
\begin{eqnarray*}
&& \bigl| \partial^a \bigl(\Gamma^n\star\widetilde
g_n \bigr) (x) - \partial^a g(x) \bigr|^2
\\
&&\qquad \le C \biggl\{ \biggl| \sum_{i,j,k} \int_{A_{i,j,k}}
\Gamma^n(x-y) \bigl| \partial^a \widetilde g_n
(y) - \partial^a \widetilde g_n(x) \bigr| \,\dd y
\biggr|^2
\\
&&\hspace*{125pt}{}
 + \bigl| \partial^a \widetilde g_n(x) -
\partial^a g(x) \bigr|^2 + S_3^a
\biggr\}
\\
&&\qquad \le C \bigl( \bigl\|D^{|a|+1} g \bigr\|_{L^\infty(G)}^2
\delta^2 +S_3^a \bigr).
\end{eqnarray*}
For $a=0$, $S_3^a=0$, but in the general case, the estimate of the term
$S_3^a$ is a little bit more delicate, since one gets more boundary
terms. In particular, when integrating by parts, derivatives along
different directions result in terms containing discontinuities of
$\widetilde g_n$ and its derivatives along the faces (that we denote
for brevity $\partial A$), edges (denoted $\partial^2 A$) or vertices
(denoted $\partial^3 A$) of the cubes, while multiple derivatives along
the same direction result in derivatives of $\Gamma^n$ appearing. For
example, for $a=(1,1,1)$ we get three kinds of terms,
\begin{eqnarray*}
&\displaystyle\int_{\partial^3 A_{i,j,k}} \Gamma^n (x-y) \Delta\widetilde
g_n(y) \,\dd y,&
\\
&\displaystyle\int_{\partial^2 A_{i,j,k}}\Gamma^n (x-y) \Delta\partial\widetilde g_n (y) \,\dd y,&
\\
&\displaystyle\int_{\partial A_{i,j,k}} \Gamma^n (x-y) \Delta
\partial^2 \widetilde g_n (y) \,\dd y,&
\end{eqnarray*}
where $\Delta$ denotes the jump of the function. For $a=(3,0,0)$ we
also have terms like
\[
\int_{\partial_1 A_{i,j,k}} \partial_1^2 \Gamma^n (x-y) \Delta\widetilde g_n(y) \,\dd y,\qquad
\int_{\partial_1 A_{i,j,k}} \partial_1 \Gamma^n (x-y) \Delta\partial_1\widetilde g_n(y) \,\dd y
\]
and in the general case we find also terms like
\[
\int_{\partial^2 A_{i,j,k}} \partial\Gamma^n (x-y) \Delta
\widetilde g_n(y) \,\dd y.
\]
However, we can bound all this terms in the same way. We have that, for
$m\in\mathbb{N}$ and $b,c\in\mathbb{N}^3$,
\begin{eqnarray*}
\int_{\partial^{m} A_{i,j,k}} \partial^{b} \Gamma^n(x)
\,\dd x &\le& C_\Gamma n^{m+|b|} = C \delta^{-(m+|b|)},
\\
\bigl|\Delta\partial^c \widetilde g_n(y) \bigr| &=& \bigl|
\partial^c\widetilde g_n \bigl(y^+ \bigr) -
\partial^c\widetilde g_n \bigl(y^- \bigr) \bigr| \le 2
\bigl\|D^4 g\bigr\|_{L^\infty(G)} \delta^{4-|c|}.
\end{eqnarray*}
Therefore
\[
S_3^a \le C \bigl\| D^4 g \bigr\|_{L^\infty(G)}^2
\delta^{2(4-|a|)}.
\]
Note that for all the terms composing $S_3^a$ we always have $m+|b|+|c|
= |a|\le3$. Summing up, we have obtained the bound
\begin{eqnarray*}
\| g - g_n \|_{\mathcal{H}}^2 &=& \sum
_a \int_{G} \bigl| \partial^a
\bigl(\Gamma^n\star\widetilde g_n \bigr) (x) -
\partial^a g(x) \bigr|^2 \,\dd x
\\
&\le& C \sum_a \bigl( \bigl\|D^{|a|+1} g
\bigr\|_{L^\infty(G)}^2 \delta^2 + \bigl\| D^4 g
\bigr\|_{L^\infty(G)}^2 \delta^{2(4-|a|)} \bigr)
\\
&\le& C \frac{1}{n^2} \| g \|_{\mathcal{C}^4(G)}^2.
\end{eqnarray*}
The lemma is proved.
\end{pf}

%
\begin{cor} \label{cor approx H}
For any $h\in\mathcal{H}$, $\lim_{n\to\infty} \pi_{\mathcal
{H}_n}^{} h = h$.
\end{cor}

\begin{pf} Fix any $\varepsilon>0$. By density, there exist a
$h_\varepsilon\in\mathcal{C}^4(G)$ s.t. $\|
h-h_\varepsilon\|_{\mathcal{H}} \le\varepsilon/2$. Also, by the
continuity of the projection, $\|\pi_{\mathcal{H}_n}^{} h - \pi
_{\mathcal{H}_n}^{} h_\varepsilon\|_{\mathcal{H}} \le\|
h-h_\varepsilon\|_{\mathcal{H}} \le\varepsilon/2$. Since $\| \pi
_{\mathcal{H}_n}^{} h_\varepsilon- h_\varepsilon\|_{\mathcal{H}} \le\|
h_{\varepsilon,n} - h_{\varepsilon} \|_\mathcal{H}$, by the above lemma
we get
\begin{eqnarray*}
\| \pi_{\mathcal{H}_n}^{} h - h \|_{\mathcal{H}} & \le&\bigl\|
\pi_{\mathcal{H}_n}^{} ( h - h_\varepsilon) \bigr\|_{\mathcal{H}} + \|
\pi_{\mathcal{H}_n}^{} h_\varepsilon- h_\varepsilon
\|_{\mathcal{H}} + \| h_\varepsilon- h \|_{\mathcal{H}}
\\
& \le&\varepsilon+ C \frac{1}{n} \|h_{\varepsilon}\|_{\mathcal
{C}^4(G)}.
\end{eqnarray*}
Therefore
\[
\limsup_{n\to\infty} \| \pi_{\mathcal{H}_n}^{} h - h
\|_{\mathcal{H}} \le\varepsilon
\]
and since $\varepsilon$ is arbitrary, the corollary is proved.
\end{pf}

%
\begin{lem}\label{lemm J zeta-sigma}
Let $F(\xi,\eta, \sigma)$ be the function defined in the proof of
Proposition \ref{prop NLS 1}. Then, whenever $\zeta_0=i\eta_0$ is the
value corresponding to a soliton component of the solution of the
deterministic NLS equation, the determinant of the Jacobian matrix
\[
J:= \pmatrix{ \partial_\xi\Re(F) & \partial_\eta\Re(F)
\cr
\partial_\xi\Im(F) & \partial_\eta\Im(F) }
\]
at point $(0,\eta_0,0)$ is not zero.
\end{lem}

\begin{pf} 
For $\sigma=0$ system (\ref{ZSSP - pert}) becomes deterministic, and
the solution is given by (\ref{NLSE det sq wall - sol1a})--(\ref{NLSE
det sq wall - sol1b}). Then, setting $c:= \sqrt{q^2-\zeta^2}$,
\[
i \partial_\xi\psi_1 (\xi,\eta,0)= \partial_\eta
\psi_1 (\xi,\eta,0)= \biggl[ \frac{q^2}{c^3} - iR
\frac{\zeta}{c} \biggr] \sin( c R ) + R \frac{\zeta^2}{c^2} \cos(c R )
\]
so that
\[
i \partial_\xi F (0,\eta_0,0) = \partial_\eta
F (0,\eta_0,0).
\]
We are left to verify that
\begin{eqnarray*}
0 \neq\det J (0,\eta_0,0) &=& \bigl[\partial_\xi\Re(F)
\partial_\eta\Im(F) - \partial_\eta\Re(F)
\partial_\xi\Im(F) \bigr] (0,\eta_0,0)
\\
&=& \bigl[\Re(\partial_\xi F ) \Im( \partial_\eta F ) -
\Re(\partial_\eta F ) \Im( \partial_\xi F ) \bigr] (0,
\eta_0,0)
\\
&=& \bigl[ \bigl(\Re( \partial_\xi F ) \bigr)^2 + \bigl(
\Im( \partial_\xi F ) \bigr)^2 \bigr] (0,
\eta_0,0)
\end{eqnarray*}
or equivalently
%
%
\begin{equation}
\label{NLS pert cond 1} \partial_\xi F (0,\eta_0,0) = -i
\biggl[ \frac{q^2}{c_0^3} + R \frac{\eta_0}{c_0} \biggr] \sin( c_0 R
) +i R\frac{\eta_0^2}{c_0^2} \cos(c_0 R ) \neq0,
\end{equation}
where $c_0=\sqrt{q^2-\eta_0^2}$. Observe that condition (\ref{NLSE det
sq wall f}) implies that $\eta_0\le q$; $c_0$ is therefore real.
Indeed, for $\eta_0>q$, $c_0$ would be purely imaginary, and the
function $f$ of equation (\ref{NLSE det sq wall f}) would become the
sum of two purely imaginary terms of the same sign, so that it cannot
be zero. In equation (\ref{NLS pert cond 1}) the coefficient of the
sinus is the sum of two nonzero terms of the same sign, so that to
ensure the condition we need to check that
%
%
\begin{equation}
\label{NLS tan-g} \tan(c_0 R) = \frac{R\eta_0^2 c_0}{q+R \eta_0 c_0^2}
\end{equation}
does \textit{not} holds for $\eta_0$ solution of (\ref{NLSE det sq wall
f}). The compatibility condition between (\ref{NLSE det sq wall
f})~and~(\ref{NLS tan-g}) is
\[
- \frac{c_0}{\eta_0}= \frac{R\eta_0^2 c_0}{q+R \eta_0 c_0^2}
\]
or equivalently
\[
\bigl( q + R\eta_0 c_0^2 + R
\eta_0^3 \bigr) c_0 =0,
\]
which cannot be satisfied (recall that $\eta_0\neq q$, so that
$c_0\neq
0$). The lemma is proved.
\end{pf}


%
\begin{prop}\label{prob reg zeta-sigma}
The stochastic differential\vspace*{-1pt} equation (\ref{ZSSP - pert}) defines a
sto\-chastic flow $\Psi_x^{(\zeta,\sigma)}=(\psi_1, \psi_2)^T$ of
$\mathcal{C}^1$-diffeomorphisms,\vspace*{-1pt} which is $\mathcal{C}^1$ also in the
parameters
$\xi,\eta,\sigma$.
\end{prop}

\begin{pf}
Write the SDE in It\^o and vector form,
%
%
\begin{equation}
\label{NLS vector} \dd\Psi= i \pmatrix{ -\zeta+i\sigma^2 & q
\cr
q &
\zeta+i\sigma^2 } \Psi\,\dd x + i \sigma\pmatrix{ 0 & 1
\cr
1 & 0 }
\Psi\,\dd W_x.
\end{equation}
The coefficients of the SDE are independent of $x$ and Lipschitz
continuous in $\Psi$ for every $\zeta$ and $\sigma$. Therefore, the
existence of a unique solution to the SDE, which defines a stochastic
flow of homeomorphisms $\Psi^{(\zeta,\sigma)}(x)$, is a classical fact;
see, for example, \cite{Ku84}. Following the notation of \cite{Ku}, we
define the local characteristic of the SDE as $(a,b,x)$, where
\begin{eqnarray*}
a \bigl(\zeta,\zeta',\sigma,\sigma',x \bigr) &:=& -
\sigma\sigma'\pmatrix{ 0 & 1
\cr
1 & 0 },
\\
b(\zeta,\sigma,x) &:=& \pmatrix{ -\zeta& q
\cr
q & \zeta}.
\end{eqnarray*}
Fix any $n\in\mathbb{N}$, define the set $G_n:=\{ (\xi,\eta,\sigma) |
|\xi|<n, 0<\eta<n, 0<\sigma<n\}$ and consider the SDE only with
parameters in $G_n$. Then, both $a$ and $b$ are uniformly bounded and,
together with their first derivatives, are Lipschitz continuous in the
parameters. This means that the coefficients satisfy condition
$(A.5)_{1,0}$ of \cite{Ku}, Chapter~4.6. It follows from
\cite{Ku}, Theorem~4.6.4, that $\Psi^{(\zeta,\sigma)}(x)$ is $\mathcal{C}^1$ in the
parameters almost surely on $G_n$. Let $\Omega_n$ be the set of
$\omega\in\Omega$ such that $\Psi^{(\zeta,\sigma)} (x) \in
\mathcal{C}^1(G_n)$: it is a set of full measure. Since $n$ is
arbitrary, $\Psi^{(\zeta,\sigma)}(x)$ is actually $\mathcal{C}^1
(\mathbb{R}\times\mathbb{R}_+\times\mathbb{R}_+ )$ for every
$\omega\in\bigcap_n \Omega_n$, which is still a set of full measure.
This proves the last statement of the proposition.

Since \cite{Ku}, Theorem 4.6.5, states that $\Psi^{(\zeta,\sigma)}(x)$
is actually a stochastic flow of~$\mathcal{C}^1$-diffeomorphisms, the
proof is complete.
\end{pf}

%
\begin{lem}\label{lem J KdV}
Let $F(\eta, \sigma)$ be the function defined in the proof of
Proposition~\ref{prop KdV 1}. Then, for all $\eta_0$ corresponding to
soliton components of solutions of the deterministic problem,
$\partial_\eta F(\eta_0,0)\neq0$.
\end{lem}

\begin{pf} 
For $\eta=\eta_{0}$ and $\sigma=0$ we have that $\varphi=\varphi_{0}$
and (recall that $c_0=\sqrt{q-\eta_0^2}$)
\begin{eqnarray*}
\partial_{\eta}\varphi(R) &=& \sin(c_0 R) \biggl[
\frac{1+ \eta_0 R}{c_0} + \frac{\eta_0^2} {c_0^3} \biggr] - \frac
{R \eta_0^2}{c_0^2} \cos
(c_0 R),
\\
\partial_\eta\widetilde\varphi(R)&=& \cos(c_0 R) [ 1+ R
\eta_0 ] + \frac{\eta_0}{c_0}\sin(c_0 R) [1 +
\eta_0 R ] = [1 + \eta_0 R ] \varphi_0 (R).
\end{eqnarray*}
Since
\begin{eqnarray*}
\partial_{\eta} F &=& \partial_{\eta} \widetilde\varphi+
\varphi+ \eta\partial_{\eta} \varphi,
\\
\partial_{\eta} F (\eta_0, 0) &=& \cos(c_0 R)
\biggl[ 2+\eta_0 R - \frac{\eta_0^3 R}{c_0^2} \biggr] + \sin(c_0
R) \biggl[ \frac{3\eta_0 + 2\eta_0^2 R}{c_0} + \frac{\eta
_0^3}{c_0^3} \biggr].
\end{eqnarray*}
The coefficient of the sinus is strictly positive (the coefficient of
the cosinus has instead at least one zero for $\eta_0\in[0,\sqrt{q}]$,
since it is positive for $\eta_0=0$ and negative for
$\eta_0\to\sqrt{q}$). Therefore, we only need to verify that the
equation%
%
\begin{equation}\label{eq tan-g}
g(R,q,\eta_0):=\tan(c_0 R) +
\frac{ 2+\eta_0 R - \eta_0^3 R/c_0^2}{(3\eta_0 + 2 \eta_0^2 R)/c_0
+ \eta_0^3/c_0^3 } =0
\end{equation}
is \textit{not} satisfied, knowing that $\eta_0$ is a value
corresponding to a soliton solution of the deterministic equation. As
we have seen, we can either have $\eta_0=\sqrt{q/2}$ if condition
(\ref{sqrt(q/2)}) is satisfied, or else $\eta_0$ is given as the
solution of equation (\ref{eq tan-f}) in $(0,\sqrt{q})\setminus\{
\sqrt{q/2}\}$.

In the first case we have that $c_0 = \sqrt{q/2}$, so that condition
(\ref{sqrt(q/2)}) implies that $\cos(c_0 R)=0$ and $\sin(c_0 R)=\pm
1$. Therefore, $\partial_\eta F\neq0$.

Consider now the second case. We look for points
$\eta\in(0,\sqrt{q})\setminus\{\sqrt{q/2}\}$ such that
$f(\eta)=g(\eta)=0$. If such a point exists, then
\[
\frac{ 2 \eta c}{q-2\eta^2} = - \frac{ 2+\eta R - \eta^3 R/c^2
}{(3\eta+ 2 \eta
^2 R)/c + \eta^3/c^3 },
\]
which also reads
\[
\frac{ 2 \eta}{q-2\eta^2} = - \frac{ (2+\eta R )c^2- \eta^3 R }{ (3\eta
+2 \eta^2 R)c^2+ \eta
^3 }
\]
or equivalently
%
%
\begin{equation}\label{test}
\frac{ (2\eta^2+2\eta^3R)(q-\eta^2) + 2\eta^4
(1+\eta R) + (2+\eta R) (q-\eta^2)q - \eta^3 R q} {
(q-2\eta^2)[ (3\eta+2 \eta^2 R)(q-\eta^2)+ \eta^3 ]}=0.\hspace*{-32pt}
\end{equation}
For $\eta< \sqrt{q/2}$ the denominator is positive and the numerator
\begin{eqnarray*}
&& \bigl(2\eta^2+2\eta^3R \bigr) \bigl(q-
\eta^2 \bigr) + 2\eta^4 (1+\eta R) + (2+\eta R) \bigl(q-
\eta^2 \bigr)q - \eta^3 R q
\\
&&\qquad > \eta R \bigl(q-\eta^2 \bigr) - \eta^3 R q >
0,
\end{eqnarray*}
so that fraction (\ref{test}) is positive and cannot be zero. For $
\sqrt{q/2}< \eta< \sqrt{q}$ the denominator is negative and the
numerator is larger than $2\eta^5 R - \eta^3 R q = \eta^3 R (\eta^2-q)
>0$, so that fraction (\ref{test}) is negative and cannot be zero. The
lemma is proved.
\end{pf}
\end{appendix}

\section*{Acknowledgments}
I wish to thank my supervisor Professor Josselin Garnier for all the
fruitful discussions that have helped me to greatly improve these results.


%

\printaddresses

\end{document}